\documentclass[12pt]{article}
\usepackage{amssymb,amsthm,tikz,tikz-cd, amsmath}
\usepackage[utf8]{inputenc}
\usepackage{mathrsfs}
\usepackage{hyperref}
\theoremstyle{definition}

\usepackage[style=numeric]{biblatex}
\newtheorem{Definition}{Definition}
\newtheorem{Proposition}{Proposition}
\newtheorem{Lemma}{Lemma}
\newtheorem{Corollary}{Corollary}

\newtheorem{Theorem}{Theorem}
\newtheorem{Example}{Example}
\newtheorem{Remark}{Remark}

\newcommand{\mat}{\text{Mat}}
\newcommand{\bA}{\mathsf{A}}
\newcommand{\bT}{\mathsf{T}}
\newcommand{\qm}{\textnormal{\textsf{QM}}}

\newcommand{\llambda}{{\boldsymbol \lambda}}
\newcommand{\mmu}{{\boldsymbol \mu}}

\newcommand{\dv}{\mathsf{v}}
\newcommand{\dw}{\mathsf{w}}

\newcommand{\ns}{\text{ns }}
\newcommand{\ahat}{\hat{a}}

\newcommand{\vrs}{\hat{\mathcal{O}}_{\text{vir}}}
\DeclareMathOperator{\Ima}{Im}

\title{On the vertex functions of type $A$ quiver varieties}
\author{Hunter Dinkins}
\date{}

\begin{document}

\maketitle

\begin{abstract}
  The goal of this paper is to better understand the quasimap vertex functions of type $A$ Nakajima quiver varieties. To that end, we construct an explicit embedding of any type $A$ quiver variety into a type $A$ quiver variety with all framings at the rightmost vertex of the quiver. Then we consider quasimap counts, showing that the map induced by this embedding on equivariant $K$-theory preserves vertex functions. 
\end{abstract}

\tableofcontents

\section{Introduction}

The topic of this paper is type $A$ Nakajima quiver varieties, see \cite{GinzburgLectures,NakALE, NakQv}, and their quasimap vertex functions \cite{pcmilect}. Fix a natural number $m$ and consider the type $A_{m}$ quiver $Q$, which has vertex set $Q_{0}=\{1,2,\ldots,m\}$ and edges $i\to i+1$ for $1 \leq i \leq m-1$. For a choice of $\dv,\dw \in \mathbb{Z}_{\geq 0}^{Q_{0}}$, called the dimension vector and framing vector respectively, and for a stability parameter $\theta \in \mathbb{Z}^{Q_{0}}$, Nakajima defined a quasiprojective algebraic variety $\mathcal{M}_{\theta}(\dv,\dw)$. These varieties are of crucial importance in geometric representation theory, see for example \cite{MO, NakQv, Naktp, pcmilect, varagnolo}. From a physical perspective, they arise as the Higgs branch of the moduli space of vacua of 3d $\mathcal{N}=4$ gauge theories, which leads to their relevance in the phenomenon of 3d mirror symmetry studied, for example, in \cite{AOElliptic, dinkms1, msflag, dinksmir4, KS1, KS2, MirSym1, MirSym2, RSbows}. To be sure, Nakajima varieties are well-defined for any choice of quiver $Q$, but our interest in this paper is only in the type $A$ setting. 

\subsection{Embedding of quiver varieties}

Assume now that $\dw_{i} \neq 0$ for some $i<m$. Let $k\in Q_{0}$ be the maximal vertex  such that $\dw_{k+1} \neq 0$. We define $\dv', \dw' \in \mathbb{Z}_{\geq 0}^{Q_{0}}$ by 

\begin{align*}
    \dv'_{i}&=\begin{cases}
        \dv_{i} & 1 \leq i \leq k \\
        \dv_{i}+i-k-1 & k+1 \leq i \leq m
    \end{cases} \\
    \dw'_{i}&=\begin{cases}
        \dw_{k+1}-1 & i=k+1 \\
        \dw_{m}+m-k & i=m \\
         \dw_{i} & \text{otherwise}
    \end{cases}
\end{align*}
In effect, $\dw'$ has one less framing than $\dw$ away from the last vertex, and the price paid for this is to change the dimension vectors and add framings at the last vertex.

\begin{Theorem}[Theorem \ref{embedthm}]\label{introthm1}
    There exists an embedding 
    \begin{equation}\label{introembed}
\Phi: \mathcal{M}_{\theta}(\dv,\dw) \hookrightarrow \mathcal{M}_{\theta}(\dv',\dw')
    \end{equation}
    Furthermore, let $\bT$ and $\bT'$ be the natural tori acting on $\mathcal{M}_{\theta}(\dv,\dw)$ and $\mathcal{M}_{\theta}(\dv',\dw')$, see section \ref{torussection}. Then there is an inclusion $\iota: \bT \hookrightarrow \bT'$ such that (\ref{introembed}) is $\bT$-equivariant. 
\end{Theorem}

We construct (\ref{introembed}) in sections \ref{seclocal} and \ref{embeddings}. Recall that $\mathcal{M}_{\theta}(\dv,\dw)$ is the moduli space of $\theta$-semistable representations of the doubled framed quiver of $Q$ with dimension vectors $\dv$ and $\dw$. In section \ref{seclocal}, we consider the data of a quiver representation corresponding to the vertices $\{k+1,k+2,\ldots,m\} \subset Q_{0}$ and we define the map (\ref{introembed}) explicitly in terms of this data. We then check that it respects the stability condition, descends to the quiver varieties, and is $\bT$-equivariant. 

Applying Theorem \ref{introthm1} repeatedly, we can embed $\mathcal{M}_{\theta}(\dv,\dw)$ into a quiver variety with all framings at the last vertex.

\begin{Corollary}\label{introcor}
    Given $\dv, \dw \in \mathbb{Z}^{Q_{0}}_{\geq 0}$, there exists $\dv', \dw' \in \mathbb{Z}^{Q_{0}}_{\geq 0}$ where $\dw'_{i}=0$ for $i \neq m$ and a $\bT$-equivariant embedding
      \begin{equation}\label{introembedflag}
\mathcal{M}_{\theta}(\dv,\dw) \hookrightarrow \mathcal{M}_{\theta}(\dv',\dw')
    \end{equation}
\end{Corollary}

If $\dw=(0,0,\ldots,N)$, it is known that the corresponding Nakajima variety $\mathcal{M}_{\theta}(\dv,\dw)$ is nonempty if and only if $\dv_{1} \leq \dv_{2} \leq \ldots \leq \dv_{m} \leq N$, \cite{NakALE} section 7. Furthermore, if $\theta=\pm (1,1,\ldots,1)$, then $\mathcal{M}_{\theta}(\dv,\dw)$ is the cotangent bundle of a partial flag variety. It follows immediately from Corollary \ref{introcor} that any type $A$ quiver variety arising from the stability conditions $\theta^{\pm}=\pm(1,1,\ldots,1)$ can be embedded into the cotangent bundle of a partial flag variety.

\subsection{Vertex functions}

Our motivation for and construction of Theorem \ref{introthm1} was inspired by the work \cite{KZins}. To explain the connection, we must recall some aspects of the enumerative geometry of Nakajima quiver varieties. As shown in the pioneering work of Maulik-Okounkov \cite{MO} and later in \cite{pcmilect}, enumerative invariants of Nakajima varieties are deeply related to the representation theory of certain quantum groups. In the $K$-theoretic setting, the relevant enumerative theory of curves is the theory of quasimaps to a GIT quotient developed in \cite{qm}. In \cite{pcmilect}, Okounkov identifies $q$-difference equations constraining certain $K$-theoretic quasimap counts with the quantum Knizhnik-Zamolodchikov equations arising from the representation theory of quantum affine algebras \cite{KZlectures}. 

One of the key curve counts studied in \cite{pcmilect} are known as vertex functions. For $d \in \mathbb{Z}^{Q_{0}}$, let $\qm^{d}_{\ns \infty}$ denote the moduli space of degree $d$ stable quasimaps from $\mathbb{P}^{1}$ to a Nakajima variety $X$ which are nonsingular at $\infty \in \mathbb{P}^{1}$. The usual action of $\mathbb{C}^{\times}$ on $\mathbb{P}^{1}$ gives rise to the action of a torus, denoted $\mathbb{C}^{\times}_{q}$, on quasimaps. Let $K_{\bT\times\mathbb{C}^{\times}_{q}}(X)_{loc}= K_{\bT\times\mathbb{C}^{\times}_{q}}(X) \otimes_{R} \text{Frac}(R)$ where $R=K_{\bT\times\mathbb{C}^{\times}_{q}}(pt)$. Let $\text{ev}_{\infty}:\qm^{d}_{\ns \infty} \to X$ be the map given by evaluating a quasimap at $\infty$. Let $\vrs^{d}$ be the symmetrized virtual structure sheaf on $\qm^{d}_{\ns \infty}$\footnote{The \textit{symmetrized} virtual structure sheaf depends on a choice of polarization of the tangent space of X. For simplicity, we supress this aspect in the introduction.}. Then the vertex function of $X$ is defined to be

$$
V(z)=\sum_{d} \text{ev}_{\infty,*}\left( \vrs^{d} \right) z^{d}\in K_{\bT \times \mathbb{C}^{\times}_{q}}(X)_{loc}[[z]]
$$
where $d$ runs over all degrees such that $\qm^{d}_{\ns \infty}$ is nonempty and $z^{d}:=\prod_{i=1}^{m} z_{i}^{d_{i}}$. The variables $z_{i}$ for $1\leq i \leq m$ can be though of as formal parameters and are conventionally referred to as ``K\"ahler paramter" (or Fayet–Iliopoulos parameters in the physics literature). In section 7 of\cite{pcmilect}, Okounkov shows that $V(z)$ satisfies a system a scalar $q$-difference equations with regular singualarities, from which it follows that $V(z)$ is in fact the Taylor series expansion of a meromorphic function of $z$. The notation $[[z]]$ above refers to a completion of the semigroup algebra of the cone of effective curves in $X$. We review vertex functions in section \ref{secqm}, but see also \cite{dinkinsthesis, dinksmir3, Pushk1, pcmilect} for further explanations.

In \cite{OkBethe}, Aganagic and Okounkov show how the Bethe equations can be obtained from quiver varieties. This was also explored in \cite{Pushk1}, where the Bethe equations for the XXZ spin chain were identified with the criticality conditions for the saddle point approximation of a contour integral computing the vertex functions for the cotangent bundle of the Grassmannian. The work of Koroteev-Zeitlin in \cite{KZins} studied 3d mirror symmetry from the perspective of Bethe equations. There the authors show that the Bethe equations associated to $\mathcal{M}_{\theta}(\dv',\dw')$ recover those associated to $\mathcal{M}_{\theta}(\dv,\dw)$ under a certain specialization of the parameters. Because the Bethe equations can be recovered from the vertex functions, which depend on the geometry of the quiver variety, it seemed desireable to us to obtain a direct geometric relationship between $\mathcal{M}_{\theta}(\dv,\dw)$ and $\mathcal{M}_{\theta}(\dv',\dw')$, hence Theorem \ref{introthm1}. Furthermore, one hopes that Theorem \ref{introthm1} would relate vertex functions, thus giving a broader explanation for the coincidences observed in \cite{KZins}.

In sections \ref{secqm} and \ref{proof}, we pursue this agenda. Let $\theta=\theta^{-}$ and consider the pullback on equivariant $K$-theory
$$
\Phi^{*}: K_{\bT'}(\mathcal{M}_{\theta}(\dv',\dw')) \to K_{\bT}(\mathcal{M}_{\theta}(\dv,\dw))
$$
induced by $\Phi$. 

% For any $\lambda \in \mathcal{M}_{\theta}(\dv,\dw)^{\bT}$, we define in section \ref{fpmap} a distinguished fixed point $\lambda' \in \mathcal{M}_{\theta}(\dv',\dw')^{\bT'}$ contained in the $\iota(\bT)$-fixed component of $\mathcal{M}_{\theta}(\dv',\dw')$ containing $\Phi(\lambda)$.

Let $V(z) \in K_{\bT \times \mathbb{C}^{\times}_{q}}(\mathcal{M}_{\theta}(\dv,\dw))_{loc}[[z]]$ and $V'(z) \in K_{\bT'\times  \mathbb{C}^{\times}_{q}}(\mathcal{M}_{\theta}(\dv',\dw'))_{loc}[[z]]$ be the vertex functions of $\mathcal{M}_{\theta}(\dv,\dw)$ and $\mathcal{M}_{\theta}(\dv',\dw')$, respectively. 

% Consider the composition
% $$
% K_{\bT'\times\mathbb{C}^{\times}_{q}}(\mathcal{M}_{\theta}(\dv',\dw')) \xrightarrow{\iota^{*}} K_{\bT \times \mathbb{C}^{\times}_{q}}(\mathcal{M}_{\theta}(\dv',\dw')) \xrightarrow{\Phi^{*}} K_{\bT\times\mathbb{C}^{\times}_{q}}(\mathcal{M}_{\theta}(\dv,\dw))
% $$

Then our main theorem is the following.

\begin{Theorem}[Theorem \ref{mainthm}]\label{mainthmintro}
    The map $\Phi^{*}$ preserves vertex functions. More precisely,
 %    $$
 % \iota^{*}(V'(z)|_{\lambda'})=V(\tilde{z})|_{\lambda}
 %    $$
 $$
\Phi^{*}(V'(z))=V(\tilde{z})
 $$
    where $\tilde{z}$ stands for a shift of the parameters $z_{1},\ldots, z_{m}$ by certain powers of $q$.
\end{Theorem}

Skipping ahead to Theorem \ref{vertex}, the reader can see that in the $K$-theoretic fixed point basis, vertex functions are certain $q$-hypergeometric series. In fact, vertex functions generalize many of the most important $q$-hypergeometric series. For the quiver variety $T^*\mathbb{P}^{n}$, one recovers the so-called $_{n+1} \phi_{n}$ basic hypergeometric series. Concretely, Theorem \ref{mainthmintro} gives a relationship between two different $q$-series under a parameter specialization. From this perspective, Theorem \ref{mainthmintro} demonstrates how then geometry of quiver varieties can be exploited to give a deeper understanding of certain special functions. Special cases of vertex functions were studied in \cite{dinksmir3,dinksmir, dinksmir4,dinksmir2, tRSKor, KorZeit,smirnovrationality}, which considered summation formulas, symmetries under swaps of the parameters (i.e. 3d mirror symmetry), and connections with Macdonald theory. 

Theorem \ref{mainthmintro} is proven in section \ref{proof} by a careful analysis of the localization formula for the vertex. Although the proof involves complicated combinatorial expressions, Theorem \ref{mainthmintro} is actually a ``term-by-term" result. By this we mean the following. Each term in the localization formula for $V'(z)$ corresponds to a $\bT' \times \mathbb{C}^{\times}_{q}$ fixed quasimap to $\mathcal{M}_{\theta}(\dv',\dw')$, and similarly for $V(z)$. Applying $\Phi^{*}$ to $V'(z)$, some of these terms become zero, and the remaining terms can be matched in a one-to-one fashion with the terms of $V(z)$. 

For this reason, one is tempted to say that $V'(z)$ is a more complicated series than $V(z)$. However, in the special case where $\mathcal{M}_{\theta}(\dv', \dw')$ is the cotangent bundle to a complete flag variety, there is at least one aspect of $V'(z)$ which is better understood: the $q$-difference equations in the K\"ahler parameters. In a future work, we will exploit this fact to prove 3d mirror symmetry of the vertex functions for those cotangent bundles of partial flag varieties whose 3d mirror duals are still Nakajima quiver varieties.

Shortly after the first version of this paper was posted, the work \cite{RB} of Rim\'anyi and Botta appeared. In the more general setting of bow varieties, they independently arrived at the same construction as Theorem \ref{introthm1}, which they call the ``D5 resolution". They prove an analog of Theorem \ref{mainthmintro} for elliptic stable envelopes and use it, along with other techniques, to prove 3d mirror symmetry for elliptic stable envelopes.

\subsection{Acknowledgements}

We would like to thank Andrey Smirnov, Peter Koroteev, and Anton Zeitlin for helpful discussions which contributed to the ideas of this paper. This project also benefited from conversations with Joshua Wen. We thank Rich\'ard Rim\'anyi and Tommaso Botta for pointing out a mistake in the first version of this paper. This research was partially supported through the NSF RTG grant Algebraic Geometry and Representation Theory at Northeastern University DMS–1645877.

\section{Review of quiver varieties}

\subsection{Definition}
We review the construction of Nakajima quiver varieties from \cite{NakALE, NakQv}, see also \cite{GinzburgLectures}. Since our interest is only in type $A$ quiver varieties, we will specialize to that case. Consider a quiver $Q$ with vertices $Q_{0}=\{1,2,\ldots,n\}$ and edges from $i$ to $i+1$ for $1 \leq i \leq n-1$. Choose $\dv,\dw \in \mathbb{Z}_{\geq 0}^{Q_{0}}$. Let $\theta\in \mathbb{Z}^{Q_{0}}$ be the stability parameter. 

For each $i \in Q_{0}$, let $V_{i}$ and $W_{i}$ be complex vector spaces of dimension $\dv_{i}$ and $\dw_{i}$, respectively. Let
$$
\text{Rep}_{Q}(\dv,\dw)= \bigoplus_{i=1}^{n-1} \text{Hom}(V_{i},V_{i+1}) \oplus \bigoplus_{i=1}^{n} \text{Hom}(W_{i},V_{i})
$$
and
$$
G_{\dv}=\prod_{i \in Q_{0}} GL(V_{i})
$$

Since $G_{\dv}$ acts on $\text{Rep}_{Q}(\dv,\dw)$ by change of basis, there is an induced Hamiltonian action of $G_{\dv}$ on $T^*\text{Rep}_{Q}(\dv,\dw)$, with associated moment map
$$
\mu_{\dv,\dw}: T^{*}\text{Rep}_{Q}(\dv,\dw) \to \text{Lie}(G_{\dv})^{*}
$$

The associated Nakajima quiver variety is defined as the algebraic symplectic reduction
\begin{equation}\label{qvdef}
\mathcal{M}_{\theta}(\dv,\dw):= T^{*}\text{Rep}_{Q}(\dv,\dw)/\!\!/\!\!/\!\!/_{\theta} G_{\dv}:= \mu_{\dv,\dw}^{-1}(0)/\!\!/_{\theta} G_{\dv}
\end{equation}
Here the notation $/\!\!/_{\theta}$ stands for the GIT quotient with stability parameter $\theta$, or more precisely, the character of $G_{\dw}$ given by
$$
(g_{i})_{i \in Q_{0}} \mapsto \prod_{i \in Q_{0}} \det(g_{i})^{-\theta_{i}}
$$

By the trace pairing, we have $\text{Hom}(A,B)^{*} \cong \text{Hom}(B,A)$, so that we can denote a general element of 
$$
T^{*}\text{Rep}_{Q}(\dv,\dw)\cong \text{Rep}_{Q}(\dv,\dw)\oplus \text{Rep}_{Q}(\dv,\dw)^{*}
$$
by a quadruple $(\{X_{i}\}_{1 \leq i \leq n-1}, \{Y_{i}\}_{1 \leq i \leq n-1}, \{I_{i}\}_{i \in Q_{0}}, \{J_{i}\}_{i \in Q_{0}})$, where $X_{i} \in \text{Hom}(V_{i},V_{i+1})$, $Y_{i} \in \text{Hom}(V_{i+1},V_{i})$, $I_{i} \in \text{Hom}(W_{i},V_{i})$ and $J_{i} \in \text{Hom}(V_{i},W_{i})$, see Figure \ref{qv}. We abbreviate this by $(X,Y,I,J)$.

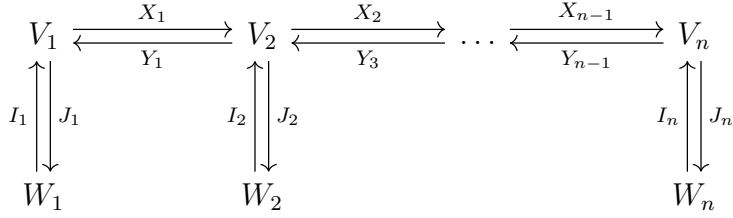
\begin{figure}[ht]
\centering
   \begin{tikzcd}[column sep=huge, row sep=huge, ampersand replacement=\&]
V_1 \arrow[d, shift left=0.5ex,"J_1"] \arrow[r,shift left=.5ex,"X_1"] \& V_2  \arrow[l,"Y_1", shift left=.5ex] \arrow[r, shift left=.5ex,"X_2"]  \arrow[d, shift left=0.5ex,"J_2"] \& \ldots \arrow[l,shift left=0.5ex, "Y_3"] \arrow[r, shift left=0.5ex, "X_{n-1}"] \& V_{n} \arrow[d, shift left=0.5ex, "J_{n}"] \arrow[l,shift left=0.5ex,"Y_{n-1}"] \\
W_{1} \arrow[u,shift left=0.5ex,"I_1"] \&  W_{2} \arrow[u,shift left=0.5ex,"I_{2}"] \&  \& W_{n} \arrow[u, shift left=0.5ex, "I_n"]
\end{tikzcd}
\label{qv}
\caption{A graphical depiction of the data $(X,Y,I,J)$.}
\end{figure}

Under the identification $\text{Lie}(G_{\dv})^{*}\cong \text{Lie}(G_{\dv})$ given by the trace pairing, the moment map is
$$
\mu_{\dv,\dw}(X,Y,I,J)=\left(X_{i-1}Y_{i-1}-Y_{i}X_{i} + I_{i} J_{i}\right)_{i \in Q_{0}} \in \bigoplus_{i \in Q_{0}} \text{End}(V_{i}) = \text{Lie}(G_{\dv})
$$

\subsection{Criterion for semistability}

We will also need the well-known criterion for (semi)stability. To state it, we define two conditions on $(X,Y,I,J)$:
\begin{enumerate}
    \item  Let $\{S_{i}\}_{i \in I}$ be a collection of subspaces $S_{i} \subset V_{i}$ preserved by $X$ and $Y$ such that $S_{i} \subset \ker J_{i}$ for all $i$. Then
    $$
 \sum_{i \in Q_{0}} \theta_{i} \dim_{\mathbb{C}}S_{i} \leq 0
    $$
    \item Let $\{T_{i}\}$ be a collection of subspaces $T_{i} \subset V_{i}$ preserved by $X$ and $Y$ such that $T_{i} \supset \Ima I_{i}$ for all $i$. Then
    $$
\sum_{i \in Q_{0}}\theta_{i} \dim_{\mathbb{C}}T_{i} \leq \sum_{i \in Q_{0}} \theta_{i} \dim_{\mathbb{C}} V_{i}
    $$
\end{enumerate}

\begin{Proposition}[\cite{GinzburgLectures} Proposition 5.1.5]\label{ss}
A quadruple $(X,Y,I,J) \in \mu_{\dv,\dw}^{-1}(0)$ is $\theta$-semistable if and only if conditions 1 and 2 above hold.

\subsection{Torus action}\label{torussection}

There a a natural action of the torus $\bA:=(\mathbb{C}^{\times})^{|\dw|}$ on $T^{*}\text{Rep}_{Q}(\dv,\dw)$ coming from the action of $\bA$ on each $W_{i}$. It descends to an action on $\mathcal{M}_{\theta}(\dv,\dw)$. In addition, there is an action of $\mathbb{C}^{\times}$ on $T^{*}\text{Rep}_{Q}(\dv,\dw)$ given by dilation of the cotangent fibers. We denote this latter torus by $\mathbb{C}^{\times}_{\hbar}$. Let $\bT=\bA\times\mathbb{C}^{\times}_{\hbar}$
   
\end{Proposition}

\section{The embedding: local case}\label{seclocal}

Our goal is to define an embedding of one quiver variety inside another. Our embedding will be constructed locally on the quiver. For the basic situation, let $\dv \in \mathbb{Z}_{\geq 0}^{n}$ be arbitrary, but assume that $\dw_{i}=0$ if $2 \leq i \leq n-1$.

\subsection{Notations}

Points in $T^{*}\text{Rep}_{Q}(\dv,\dw)$ are given by diagrams of the form:

\begin{center}
   \begin{tikzcd}[column sep=huge, row sep=large, ampersand replacement=\&]
V_1 \arrow[d, shift left=0.5ex,"J_1"] \arrow[r,shift left=.5ex,"X_1"] \& V_2  \arrow[l,"Y_1", shift left=.5ex] \arrow[r, shift left=.5ex,"X_2"] \& \ldots \arrow[l,shift left=0.5ex, "Y_3"] \arrow[r, shift left=0.5ex, "X_{n-1}"] \& V_{n} \arrow[d, shift left=0.5ex, "J_{n}"] \arrow[l,shift left=0.5ex,"Y_{n-1}"] \\
W_{1} \arrow[u,shift left=0.5ex,"I_{1}"] \& \& \& W_{n} \arrow[u, shift left=0.5ex, "I_n"]
\end{tikzcd}
\end{center}

Choose a basis for each vector space $V_i$ and $W_i$ and write all linear maps in the diagram above as matrices. In particular, we write 
\[ I_1=
\begin{pmatrix}
A_1 & A_2 & \ldots & A_{\dw_1}
\end{pmatrix}
\]
where each $A_{k}$ is a column vector in $\mathbb{C}^{\dv_{1}}$. Similarly, we write
\[ J_1=
\begin{pmatrix}
B_1 \\
B_2 \\
\vdots \\
B_{\dw_1}
\end{pmatrix}
\]
where each $B_{k}$ is a row vector in $\mathbb{C}^{\dv_{1}}$.

Let $\dv',\dw' \in \mathbb{Z}^{n}_{\geq 0}$ be defined by
\begin{align*}
     \dv_{i}' &= \dv_{i}+i-1 \\
     \dw_{i} &= \begin{cases}
         \dw_{1}-1 & i=1 \\
         0 & 2 \leq i \leq n-1 \\
         \dw_{n}+n & i=n
     \end{cases}
\end{align*}
Let $V_{i}'$ and $W_{i}'$ be complex vector spaces of dimension $\dv_{i}'$ and $\dw_{i}'$, respectively. We identify $V_{i}'=V_{i} \oplus \mathbb{C}^{i-1}$ for $1 \leq i \leq n$ and $W_{n}'=W_{n}\oplus \mathbb{C}^{n}$. 

\subsection{Construction of the map}\label{localconstruction}

To $(X,Y,I,J)\in T^{*} \text{Rep}_{Q}(\dv,\dw)$ we associate an element $(X',Y',I',J')$ of $T^{*}\text{Rep}_{Q}(\dv',\dw')$ as follows:

\begin{itemize}
    \item The framing maps at the first node are 
    \begin{align*}
 I'_{1}=
\begin{pmatrix}
A_2 & A_3 & \ldots & A_{\dw_1}
\end{pmatrix}, \quad 
J'_{1} = \begin{pmatrix}
B_2 \\ B_3 \\ \vdots \\ B_{\dw_1}
\end{pmatrix}
\end{align*}
\item $X'_{k}: V_k \oplus \mathbb{C}^{k-1} \to V_{k+1} \oplus \mathbb{C}^{k}$ is given by
\[ X'_{k}=
\begin{tikzpicture}[baseline={([yshift=-.5ex]current bounding box.center)}]
\matrix [matrix of math nodes,left delimiter=(,right delimiter=),row sep=0.1cm,column sep=0.1cm] (m) {
      -X_k & 0 \\
      B_1 Y_1 \ldots Y_{k-1} & C_k \\
      0 & -\mathbb{I}_{k-1} \\};
\end{tikzpicture}
\]
where 
\begin{align*}
 C_k=\begin{pmatrix} c^{1} & c^{2} & \ldots & c^{k-1} \end{pmatrix} &\in \mat_{1,k-1}(\mathbb{C}),  \quad c^{j} = B_1 \left( Y_1 \ldots Y_{j-1} X_{j-1} \ldots X_{1} \right)A_1
 \end{align*}
 and $\mathbb{I}_{k-1}$ is the $(k-1)\times(k-1)$ identity matrix.

Notice that
\begin{align*} 
 X_{k} &\in \mat_{\dv_{k+1},\dv_k}(\mathbb{C}) \\
 J_1 Y_1 \ldots Y_{k-1} &\in \mat_{1,\dv_{k}}(\mathbb{C})
\end{align*}
so that $X'_{k}$ is a $(\dv_{k+1}+1+(k-1))\times(\dv_{k}+(k-1))=\dv'_{k}\times\dv'_{k+1}$ matrix.
      
\item $Y'_{k}: V_{k+1} \oplus \mathbb{C}^{k} \to V_{k} \oplus \mathbb{C}^{k-1}$ is given by
\[ Y'_{k}=
\begin{tikzpicture}[baseline={([yshift=-.5ex]current bounding box.center)}]
\matrix [matrix of math nodes,left delimiter=(,right delimiter=),row sep=0.1cm,column sep=0.1cm] (m) {
      Y_{k} & 0 & X_{k-1} \ldots X_{1} A_{1} \\ 0 & \mathbb{I}_{k-1} & 0 \\};
\end{tikzpicture}
\]
Notice that
\begin{align*}
    Y_{k} &\in \mat_{\dv_k,\dv_{k+1}}(\mathbb{C}) \\
    X_{k-1} \ldots X_1 A_1 & \in \mat_{\dv_{k},1}(\mathbb{C})
\end{align*}
so that $Y'_{k}$ is a $((\dv_k)+(k-1))\times((\dv_{k+1})+(k-1)+1)=\dv'_{k}\times\dv'_{k+1}$ matrix.

\item The new framing maps at the final node are
\[ J'_{n}=-
\begin{tikzpicture}[baseline={([yshift=-.5ex]current bounding box.center)}]
\matrix [matrix of math nodes,left delimiter=(,right delimiter=),row sep=0.1cm,column sep=0.1cm] (m) {
      J_n & 0 \\
      B_1 Y_1 \ldots Y_{n-1} & C_n \\
      0 & -\mathbb{I}_{n-1} \\};
\end{tikzpicture}
\]
and 
\[ I'_{n}=
\begin{tikzpicture}[baseline={([yshift=-.5ex]current bounding box.center)}]
\matrix [matrix of math nodes,left delimiter=(,right delimiter=),row sep=0.1cm,column sep=0.1cm] (m) {
      I_{n} & 0 & X_{n-1} \ldots X_{1} A_{1} \\ 0 & \mathbb{I}_{n-1} & 0 \\};
\end{tikzpicture}
\]
\end{itemize}

\begin{figure}
    \centering
\begin{tikzcd}[column sep=large, row sep=large,ampersand replacement=\&]
V_1 \arrow[d, shift left=0.5ex,"J_{1}'"] \arrow[r,shift left=.5ex,"X'_1"] \& V_2 \oplus \mathbb{C} \arrow[l,"Y'_1", shift left=.5ex] \arrow[r, shift left=.5ex,"X'_2"] \& \ldots \arrow[l,shift left=0.5ex, "Y'_3"] \arrow[r, shift left=0.5ex, "X'_{n-1}"] \& V_{n}\oplus \mathbb{C}^{n-1} \arrow[d, shift left=0.5ex, "J'_{n}"] \arrow[l,shift left=0.5ex,"Y'_{n-1}"] \\
\mathbb{C}^{\dw_1-1} \arrow[u,shift left=0.5ex,"I_{1}'"] \&  \& \& W_{n}\oplus \mathbb{C}^{n} \arrow[u, shift left=0.5ex, "I_{n}'"]
\end{tikzcd}
     \caption{The new quiver representation.}
    \label{repnew}
\end{figure}

This construction defines a map $\Phi: T^{*}\text{Rep}_{Q}(\dv,\dw) \to T^{*}\text{Rep}_{Q}(\dv',\dw')$.

\subsection{Induced map on quiver varieties}
We next check that the map $T^{*}\text{Rep}_{Q}(\dv,\dw) \to T^{*}\text{Rep}_{Q}(\dv',\dw')$ constructed in the last section descends to a map on quiver varieties. In the proofs below, we will freely use the notation $A_{1}$, $B_{1}$, and $C_{k}$ defined in the previous section.

\begin{Proposition}
If $(X,Y,I,J)\in \mu_{\dv,\dw}^{-1}(0)$, then $\Phi(X,Y,I,J) \in \mu_{\dv',\dw'}^{-1}(0)$.
\end{Proposition}
\begin{proof}
We denote $\Phi(X,Y,I,J)=(X',Y',I',J')$. For the first vertex, we calculate
\begin{align*}
-Y'_1 X'_1 + I'_{1} J'_{1} &=\begin{pmatrix}
Y_1 & A_1
\end{pmatrix}\cdot \begin{pmatrix}
      -X_1 \\
      B_1
      \end{pmatrix} +
      \begin{pmatrix} A_2 & A_3 &\ldots & A_{\dw_1}
      \end{pmatrix}
      \cdot \begin{pmatrix}
      B_2 \\ B_3 \\ \vdots \\ B_{\dw_1}
      \end{pmatrix} \\
      &= -Y_{1} X_{1}+ \sum_{j=1}^{\dw_{1}} A_{j} B_{j} \\
     &= -Y_1 X_1+I_{1} J_{1} =0
\end{align*}
For the last vertex, we calculate
\begin{align*}
X'_{n-1} Y'_{n-1}  &=\begin{pmatrix}
-X_{n-1} & 0 \\ B_1 Y_1\ldots Y_{n-2} & C_{n-1} \\ 0 & -\mathbb{I}_{n-2}
\end{pmatrix}\cdot \begin{pmatrix} Y_{n-1} & 0 & X_{n-2} \ldots X_{1} A_{1} \\ 0 & \mathbb{I}_{n-2} & 0
      \end{pmatrix}  \\
      &=\begin{pmatrix}
-X_{n-1} Y_{n-1} & 0 & -X_{n-1} \ldots X_1 A_1 \\ B_1 Y_1 \ldots Y_{n-1} & C_{n-1} & B_1 Y_1 \ldots Y_{n-2} X_{n-2} \ldots X_1 A_1 \\ 0 & -\mathbb{I}_{n-2} & 0
\end{pmatrix}
\end{align*} 
and
\begin{align*}
I'_n J'_n &= -\begin{pmatrix}  I_{n} & 0 & X_{n-1} \ldots X_{1} A_{1} \\ 0 & \mathbb{I}_{n-1} & 0
      \end{pmatrix}
      \cdot \begin{pmatrix}
      J_n & 0 \\
      B_1 Y_1 \ldots Y_{n-1} & C_n \\
      0 & -\mathbb{I}_{n-1}
      \end{pmatrix} \\
      &=-\begin{pmatrix} I_n J_n & 0 & -X_{n-1} \ldots X_1 A_1 \\ B_1 Y_1 \ldots Y_{n-1} & C_{n-1} & B_1 Y_1 \ldots Y_{n-2} X_{n-2} \ldots X_{1} A_{1} \\ 0 & -\mathbb{I}_{n-2} & 0
\end{pmatrix}
\end{align*} 
which shows that
\[
X'_{n-1} Y'_{n-1} + I'_{n} J'_{n} = 0
\]

Now suppose that $n>2$ and fix $k$ so that $1<k<n$. Then
\begin{align*}
X'_{k} Y'_{k}&=\begin{pmatrix}
 -X_k & 0 \\
      B_1 Y_1 \ldots Y_{k-1} & C_k \\
      0 & -\mathbb{I}_{k-1}
\end{pmatrix} \cdot \begin{pmatrix}
 Y_{k} & 0 & X_{k-1} \ldots X_{1} A_{1} \\ 0 & \mathbb{I}_{k-1} & 0
\end{pmatrix} \\
&=\begin{pmatrix}
-X_k Y_k & 0 & -X_{k} \ldots X_1 A_1 \\ B_1 Y_1 \ldots Y_{k} & C_{k} & B_1 Y_1 \ldots Y_{k-1} X_{k-1} \ldots X_1 A_1 \\ 0 & -\mathbb{I}_{k-1} & 0
\end{pmatrix}
\end{align*}
and
\begin{align*}
   Y'_{k+1} X'_{k+1}&=\begin{pmatrix}
   Y_{k+1} & 0 & X_{k} \ldots X_{1} A_1 \\ 0 & \mathbb{I}_{k} & 0
   \end{pmatrix} \cdot
   \begin{pmatrix}
   -X_{k+1} & 0 \\ B_1 Y_1 \ldots Y_{k} & C_{k+1} \\
  0 & -\mathbb{I}_{k}
   \end{pmatrix} \\
   &= \begin{pmatrix}
  - Y_{k+1} X_{k+1} & 0 & -X_{k} \ldots X_{1} A_1 \\  B_1 Y_1 \ldots Y_k  & C_k & B_1 Y_1 \ldots Y_{k-1} X_{k-1} \ldots X_1 A_1 \\ 0 & -\mathbb{I}_{k-1} & 0
   \end{pmatrix}
\end{align*}
So 
\[
X'_{k} Y'_{k} - Y'_{k+1} X'_{k+1}=0
\]
which concludes the proof.

\end{proof}

Now fix a stability parameter $\theta \in \mathbb{Z}^{n}$.
\begin{Proposition}\label{preservestability}
If $(X,Y,I,J)$ is $\theta$-semistable, then $\Phi(X,Y,I,J)$ is $\theta$-semistable.
\end{Proposition}

% \begin{Remark}
% Recall that if $\theta$ is chosen generically with respect to the dimension parameters, then the notions of stability and semistability coincide.
% \end{Remark}
\begin{proof}
We use Proposition \ref{ss}. Suppose $(X,Y,I,J) \in T^*\text{Rep}_{Q}(\dv, \dw)$ is $\theta$-semistable. As above, write $\Phi(X,Y,I,J)=(X',Y',I',J')$.

Let $\{S_{i}\}$ be a collection of subspaces $S_{i} \subset V_{i}'$ preserved by the maps $X_{k}'$ and $Y_{k}'$ such that $S_{1} \subset \ker J_{1}'$ and $S_{n} \subset \ker J_{n}'$. By definition of $J_{n}'$, this forces $S_{n} \subset V_{n} \subset V_{n} \oplus \mathbb{C}^{n-1}$. By definition of $X_{n-1}'$, it is obvious that if $X_{n-1}'(S_{n-1})  \subset V_{n}$, then $S_{n-1} \subset V_{n-1}$. Continuing inductively, we see that $S_{i} \subset V_{i}$ for all $i$. Furthermore, $X_{1}'(S_{1}) \subset V_{2}$ and $S_{1} \in \ker J_{1}'$ together imply that $S_1 \subset \ker J_{1}$. Since $(X,Y,I,J)$ is $\theta$-semistable, Proposition \ref{ss} implies that $\sum_{j=1}^{n} \theta_{j} \cdot \dim_{\mathbb{C}} S_{j} \leq 0$.

Let $\{T_{i}\}$ be a collection of subspaces $T_{i} \subset V_{i}'$ preserved by the maps $X_{k}'$ and $Y_{k}'$ such that $T_{1} \supset \Ima I'_{1}$ and $T_{n} \supset \Ima I_{n}'$. Let $T_{i}'= T_{i} \cap V_{i}$. Since $T_{i}$ is preserved by $X_{i}'$ and $Y_{i}'$, it is clear from the definition of $X_{i}'$ and $Y_{i}'$ that $T_{i}'$ is preserved by $X_{i}$ and $Y_{i}$. 

% We must show $\theta \cdot \dim_{I} T \leq \theta \cdot \dim_{I} V^{\text{new}}$. 

By definition of $I_{n}'$, it is clear that $T_{n} \supset U_{n} \oplus \mathbb{C}^{n-1}$ for some subspace $U_{n} \subset V_{n}$ and $U_{n} \supset \Ima I_{n}$. Proceeding inductively, we also see that $T_{i}\supset U_{i} \oplus \mathbb{C}^{i-1}$ for some subspace $U_{i} \subset V_{i}$ for all $i$. In particular, $T_2 \supset U_{2} \oplus \mathbb{C}$.  Since $Y_{1}'=\begin{pmatrix} Y_{1} & A_{1}\end{pmatrix}$, we see that $T_{1} \supset \Ima A_{1}$. Thus $T_{1} \supset \Ima I_1$. From $T_{1} \supset \Ima I_1$ and $T_{n} \supset \Ima I_{n}'$, we see that $T_{1}'\supset \Ima I_{1}$ and $T_{n}'\supset \Ima I_{n}$.

Thus we have
\begin{align*}
\sum_{j=1}^{n}\theta_{j} \dim_{\mathbb{C}} T_j &= \sum_{i} \theta_{i} (\dim_{\mathbb{C}} T_{i}'+i-1) \\
&\leq \sum_{j=1}^{n} \theta_{j} \dim_{\mathbb{C}} V_{j} + \sum_{i} \theta_{i}(i-1) \\
&= \sum_{j=1}^{n} \theta_{j}\dim_{\mathbb{C}} V'_{j}
\end{align*}

We have verified both conditions of Proposition \ref{ss}. Thus $(X',Y',I',J')$ is $\theta$-semistable.

\end{proof}

We denote $G_{\mathsf{v}'}= \prod_{i=1}^{n} GL(V_i \oplus \mathbb{C}^{i-1})$. We consider the Nakajima quiver varieties 
\begin{align*}
  & \mathcal{M}:= \mathcal{M}_{\theta}(\dv,\dw):= \mu_{\dv,\dw}^{-1}(0)/\!\!/G_{\dv} \\
  & \mathcal{M}':=\mathcal{M}_{\theta}(\dv',\dw')= \mu_{\dv',\dw'}^{-1}(0)/\!\!/G_{\dv'}
\end{align*}

There is a natural inclusion $\rho$ from $G_{\mathsf{v}}$ to $G_{\mathsf{v}'}$, defined by inclusion into the first component. Write $g\in G_{\dv}$ as $g=(g_{1},g_{2},\ldots,g_{n})$. Then
$$
\rho=(\rho_{1},\ldots, \rho_{n}), \quad \rho_{i}(g_{i})= \begin{pmatrix}
g_i & 0 \\ 0 & \mathbb{I}_{i-1}
\end{pmatrix}
$$
\begin{Proposition}\label{equivariant}
The map $\Phi$ is $\rho$-equivariant. In particular, $\Phi$ descends to a map $\mathcal{M} \to \mathcal{M}'$, which we also denote by $\Phi$.
\end{Proposition}

\begin{proof}
    This follows from the block form of $\Phi(X,Y,I,J)$. For example,
    \begin{align*}
    \rho_{k+1}(g_{k+1}) X_{k}' \rho_{k}(g_{k})^{-1} &= \begin{pmatrix} g_{k+1} & 0 \\ 0 & \mathbb{I}_{k} \end{pmatrix}\begin{pmatrix}
 -X_k & 0 \\
      B_1 Y_1 \ldots Y_{k-1} & C_k \\
      0 & -\mathbb{I}_{k-1}
\end{pmatrix} \begin{pmatrix} g_{k}^{-1} & 0 \\ 0 & \mathbb{I}_{k-1} \end{pmatrix} \\
&=\begin{pmatrix}
 -g_{k+1} X_k g_{k}^{-1} & 0 \\
      B_1 Y_1 \ldots Y_{k-1} g_{k}^{-1} & C_k \\
      0 & -\mathbb{I}_{k-1}
\end{pmatrix}
    \end{align*}

On the other hand, the component of $\Phi(g \cdot (X,Y,I,J))$ inside $\text{Hom}(V_{k}',V_{k+1}')$ is
\begin{align*}
    \begin{pmatrix}
        -g_{k+1}X_{k} g_{k}^{-1} & 0 \\ (B_{1} g_{1}^{-1}) (g_1 Y_1 g_2^{-1}) \ldots (g_{k-1} Y_{k-1} g_{k}^{-1}) & C_{k} \\ 0 & -\mathbb{I}_{k-1}
    \end{pmatrix} \\=\begin{pmatrix}
 -g_{k+1} X_k g_{k}^{-1} & 0 \\
      B_1 Y_1 \ldots Y_{k-1} g_{k}^{-1} & C_k \\
      0 & -\mathbb{I}_{k-1}
\end{pmatrix}
\end{align*}
The $C_{k}$ block is unchanged because of the formula
\begin{multline*}
B_{1}(Y_{1} \ldots Y_{j-1} X_{j-1} \ldots X_{1})A_{1} \\ = B_{1} g_{1}^{-1}(g_{1}Y_{1} g_{2}^{-1}\ldots g_{j-1} Y_{j-1} g_{j}^{-1} g_{j} X_{j-1} g_{j-1}^{-1} \ldots g_{2} X_{1} g_{1}^{-1})g_{1} A_{1}
\end{multline*}

Similar computations for the rest of the data show that $\Phi(g \cdot (X,Y,I,J))= \rho(g) \Phi(X,Y,I,J)$.
\end{proof}

By the previous three propositions, $\Phi$ descends to a map of quiver varieties. 

\begin{Remark}
From now on, we will use $\Phi$ to denote the induced map
\begin{equation}\label{localphi}
\Phi: \mathcal{M} \to \mathcal{M}'
\end{equation}
\end{Remark}

\subsection{Injectivity}

\begin{Proposition}\label{injective}
The map (\ref{localphi}) of quiver varieties is injective.
\end{Proposition}

\begin{proof}
We must show that if $\Phi(X,Y,I,J)$ and $\Phi(X',Y',I',J')$, are in the same $G_{\mathsf{v}'}$-orbit, then $(X,Y,I,J)$ and $(X',Y',I',J')$ are in the same $G_{\mathsf{v}}$-orbit. Note that in this proof $(X',Y',I',J')$ denotes a point in the domain of (\ref{localphi}), in contrast to previous usage.

So suppose that $g \cdot \Phi(X,Y,I,J)=\Phi(X',Y',I',J')$. We write the components $g_i$ of $g$, where $g_i \in GL(V_i\oplus \mathbb{C}^{i-1})$ in $(\mathsf{v}_i+(i-1))\times (\mathsf{v}_i+(i-1))$ block form as
\[
\begin{pmatrix}
g_{i,1} & g_{i,2} \\
g_{i,3} & g_{i,4}
\end{pmatrix}
\]

Consider the component of the equation $g \cdot \Phi(X,Y,I,J)=\Phi(X',Y',I',J')$ lying in $\text{Hom}(W_{n}',V_{n}')$:
\begin{align*}
\begin{pmatrix}
g_{n,1} & g_{n,2} \\
g_{n,3} & g_{n,4}
\end{pmatrix} \cdot  \begin{pmatrix}  I_{n} & 0 & X_{n-1} \ldots X_{1} A_{1} \\ 0 & \mathbb{I}_{n-1} & 0
      \end{pmatrix}&= \begin{pmatrix}
      g_{n,1} I_n & g_{n,2} & g_{n,1} X_{n-1} \ldots X_1 A_1 \\ g_{n,3} I_n & g_{n,4} & g_{n,3} X_{n-1} \ldots x_1 A_1
      \end{pmatrix} \\
      &=\begin{pmatrix}
      I_n' & 0 & X_{n-1}' \ldots X_1 ' I_1' \\
      0 & \mathbb{I}_{n-1} & 0
      \end{pmatrix}
\end{align*}
This implies that $g_{n,2}=0$ and $g_{n,4}=\mathbb{I}_{n-1}$. 

The inverse of $g_{n}$ is given as follows:
\[ g_{n}^{-1}=
\begin{pmatrix}
g_{n,1}& 0 \\ g_{n,3} & \mathbb{I}_{n-1}
\end{pmatrix}^{-1} = \begin{pmatrix}
g_{n,1}^{-1} & 0 \\ -g_{n,3} g_{n,1}^{-1} & \mathbb{I}_{n-1}
\end{pmatrix}
\]
We write this as a $(\mathsf{v}_n+(n-1))\times(\mathsf{v}_n+(n-2)+1)$ block matrix as
\[
\begin{pmatrix}
g_{n,1}^{-1} & 0 & 0 \\ -g_{n,3} g_{n,1}^{-1} & \mathbb{I}_{n-1,n-2} & D
\end{pmatrix}
\]
where $\mathbb{I}_{n-1,n-2}$ is the $(n-1)\times (n-2)$ matrix with $1$'s on the main diagonal and $D$ is the $(n-1)\times 1$ matrix with 1 in the last entry and 0 elsewhere.

Write the inverse of $g_n$ as 
\[g_{n}^{-1}
\begin{pmatrix}
h_{n,1} & 0 \\ h_{n,3} & \mathbb{I}_{n-1}
\end{pmatrix}
\]
and consider the component of the equation  $g \cdot \Phi(X,Y,I,J)=\Phi(X',Y',I',J')$ in $\text{Hom}(V_{n}',W_{n}')$:

\begin{align*} &-\begin{pmatrix}
  J_n & 0 \\
      B_1 Y_1 \ldots Y_{n-1} & C_n \\
      0 & -\mathbb{I}_{n-1} 
\end{pmatrix}\cdot
  \begin{pmatrix}
h_{n,1} & 0 \\ h_{n,3} & \mathbb{I}_{n-1}
\end{pmatrix} \\
&= -\begin{pmatrix}
J_n h^{n}_1 & 0 \\ B_1 Y_1 \ldots Y_{n-1} h_{n,1} + C_n h_{n,3} & C_n \mathbb{I}_{n-1} \\ - h_{n,3} & -\mathbb{I}_{n-1}
\end{pmatrix} \\
&=\begin{pmatrix}
J_n' & 0 \\ B_1' Y_1' \ldots Y_{n-1}' & C_n' \\ 0 & -\mathbb{I}_{n-1}
\end{pmatrix}
\end{align*}
So $h_{n,3}=0$, which implies that $g_{n,3}=0$.

Continuing inductively, we can show that 
\[
g_i = \begin{pmatrix}
g_{i,1} & 0\\
0 & \mathbb{I}_{i-1}
\end{pmatrix}
\]
which means that $g=\rho((g_{i,1})_{1\leq i\leq n})$. Letting $\tilde{g}=(g_{i,1})_{1 \leq i \leq n}$, we see that $\tilde{g} \cdot (X,Y,I,J)= (X',Y',I',J')$.

\end{proof}

\subsection{Torus equivariance}\label{torusmap}

The varieties $\mathcal{M}$ and $\mathcal{M}'$ are acted on by tori $\bT$ and $\bT'$, respectively. 

Denote the coordinates on torus $\bT$ by $(a_{1,1},\ldots,a_{1,\dw_1},a_{n,1},\ldots,a_{n,\dw_n}, \hbar)$ and the coordinates on torus $\bT'$ by $(b_{1,1}, \ldots, b_{1,\dw_1-1}, b_{n,1},\ldots b_{n,\dw_n},c_1,\ldots,c_{n},\hbar')$. We define the map
\[
\iota: \bT \to \bT'
\]
which is given on coordinates by
\begin{align}\label{torimap}
    &b_{1,j-1}\mapsto a_{1,j}, \quad \text{for} \quad  j \in \{2,\ldots,\dw_1\} \\ \nonumber
    &b_{n,j}\mapsto a_{n,j}, \quad \text{for} \quad  j \in \{1,\ldots,\dw_n\} \\ \nonumber
    &c_{j}\mapsto\hbar^{j-n} a_{1,1} , \quad \text{for} \quad j \in \{1,\ldots,n\} \\ \nonumber
    &\hbar'\mapsto\hbar
\end{align}
We will abuse notation and just write $\hbar$ instead of $\hbar'$. We will abbreviate elements of $\bT$ as $(a,\hbar)$ and elements of $\bT'$ as $(b,c,\hbar)$.

\begin{Proposition}\label{toruseqprop}
The map $\Phi$ is equivariant with respect to $\iota$
\end{Proposition}
\begin{proof}
We must show that $\iota(a,\hbar) \cdot \Phi(p) = \Phi((a,\hbar) \cdot p)$. Let $(X,Y,I,J)$ be a representative of $p$. Then we must show that \begin{equation}\label{toruseq}
    \Phi((a,\hbar) \cdot (X,Y,I,J)) =g (\cdot \iota(a,\hbar)  \cdot \Phi(X,Y,I,J) )
\end{equation}
for some $g \in G_{\dv'}$ (we have abused notation here, using $\Phi$ for the map on the prequotient). We will inductively define the appropriate $g$. Write each $g_{i} \in \text{End}(V_{i})$ in block form as
$$
g_{i}=\begin{pmatrix}
    g_{i,1} & g_{i,2} \\
    g_{i,3} & g_{i,4}
\end{pmatrix}
$$
Set $g_{n,1}=\mathbb{I}_{\dv_n,\dv_n}$ and $g_{n,4}=\text{diag}(\hbar^{1-n} a_{1,1},\ldots,\hbar^{-1} a_{1,1})$. Then the component of $g \cdot (\iota(a,\hbar) \cdot \Phi(X,Y,I,J))$ in $\text{Hom}(W_{n},V_{n})$ is
\begin{multline*}
\hbar^{-1} \cdot
\begin{pmatrix}
\text{diag}(a_{n,1},\ldots,a_{n,\dw_n}) J_n g_{n,1}^{-1} & 0  \\
\hbar^{1-n} a_{1,1} B_1 Y_1 \ldots Y_{n-1}g_{n,1}^{-1} & \hbar^{1-n} a_{1,1} C_n g_{n,4}^{-1} \\
0 & -\text{diag}(\hbar^{2-n}a_{1,1},\ldots,a_{1,1}) g_{n,4}^{-1}
\end{pmatrix} \\
=   \begin{pmatrix}
\hbar^{-1} \text{diag}(a_{n,1},\ldots,a_{n,\dw_n}) J_n & 0  \\
\hbar^{-n} a_{1,1} B_1 Y_1 \ldots Y_{n-1} &  \hbar^{-n} a_{1,1} C_n g_{n,4}^{-1} \\
0 & -\mathbb{I}_{n-1}
\end{pmatrix}
\end{multline*}
We calculate
\[
\hbar^{-n} a_{1,1} C_n g_{n,4}^{-1} = \begin{pmatrix}
\hbar^{-1} B_1 A_1 & \hbar^{-2} B_1 Y_1 X_1 A_1 & \ldots & \hbar^{-n+1} B_1 (Y_1 \ldots Y_{n-2} X_{n-2} \ldots X_1 ) A_1
\end{pmatrix}
\]
so that (\ref{toruseq}) holds for this component.

Similarly, the component of $g \cdot (\iota(a,\hbar) \cdot \Phi(X,Y,I,J))$ in $\text{Hom}(W_{n},V_{n})$ is
\begin{multline*}
\begin{pmatrix}
g_{n,1} I_n \text{diag}(a_{n,1},\ldots,a_{n,\dw_n})^{-1} & 0 & g_{n,1} X_{n-1} \ldots X_{1} A_{1} a_{1,1}^{-1} \\
0 & g_{n,4} \text{diag}(\hbar^{1-n} a_{1,1},\ldots,\hbar^{-1}a_{1,1})^{-1} & 0
\end{pmatrix} \\
=\begin{pmatrix}
I_n \text{diag}(a_{n,1},\ldots,a_{n,\dw_n})^{-1} & 0 & X_{n-1} \ldots X_1 A_1 a_{1,1}^{-1} \\
0 &  \mathbb{I}_{n-1} & 0
\end{pmatrix} \\
\end{multline*}
which equals the component of $\Phi((a,\hbar)\cdot (X,Y,I,J))$ in $\text{Hom}(W_{n},V_{n})$.

Next, set $g_{n-1,1}=\mathbb{I}_{\dv_{n-1},\dv_{n-1}}$, $g_{n-1,4}=\text{diag}(\hbar^{2-n}a_{1,1},\ldots,\hbar^{-1}a_{1,1})$, $g_{n-1,2}=0$, and $g_{n-1,3}=0$. The torus $\iota(\bT)$ acts trivially on the component of $\Phi(X,Y,I,J)$ in $\text{Hom}(V_{n-1},V_{n})$; so we must show that this component is preserved by the action of $g$. We compute
\begin{align*}
g_{n} X'_{n-1} g_{n-1}^{-1} &= \begin{pmatrix}
-X_{n-1} & 0 \\ \hbar^{1-n}a_{1,1} B_1 Y_1 \ldots Y_{n-2} & \hbar^{1-n} a_{1,1} C_{n-1} g_{n-1,4}^{-1} \\ 0 & \text{diag}(\hbar^{2-n} a_{1,1},\ldots, \hbar^{-1} a_{1,1}) g_{n-1,4}^{-1}
\end{pmatrix} \\
&= \begin{pmatrix}
-X_{n-1} & 0 \\ \hbar^{1-n} a_{1,1} B_1 Y_1 \ldots Y_{n-2} & \hbar^{1-n}a_{1,1} C_{n-1} g_{n-1,4}^{-1} \\ 0 & \mathbb{I}_{n-2}
\end{pmatrix} 
\end{align*}
Note that 
\[
\hbar^{1-n} a_{1,1} C_{n-1}g_{n-1,4}^{-1} = \begin{pmatrix}
\hbar^{-1} B_1 A_1 & \hbar^{-2} B_1 Y_1 X_1 A_1 & \ldots & \hbar^{-(n-2)} B_1 (Y_1 \ldots Y_{n-3} X_{n-3} \ldots X_1) A_1
\end{pmatrix}
\]
so that (\ref{toruseq}) holds for the component in $\text{Hom}(V_{n-1},V_{n})$. For the component in $\text{Hom}(V_{n},V_{n-1})$, we have
\begin{multline*}
g_{n-1} \hbar^{-1} 
\begin{pmatrix}
      Y_{n-1} & 0 & X_{n-2} \ldots X_{1} A_{1} \\ 0 & \mathbb{I}_{n-2} & 0
     \end{pmatrix}  g_n^{-1} \\ 
     = \hbar^{-1} \begin{pmatrix}
          Y_{n-1} & 0 & X_{n-2} \ldots X_1 A_1 \\ 0 & \text{diag}(\hbar^{2-n} a_{1,1},\ldots,\hbar^{-1} a_{1,1}) & 0
     \end{pmatrix}  g_n^{-1} \\
     =\begin{pmatrix}
           \hbar^{-1} Y_{n-1} & 0 &  X_{n-2} \ldots X_{1} A_1 a_{1,1}^{-1}\\ 0 & \mathbb{I}_{n-2} & 0
     \end{pmatrix}
\end{multline*}
Continuing inductively, we see that (\ref{toruseq}) holds.

\end{proof}

\begin{Proposition}
The largest subtorus of $\bT'$ that preserves $\Phi(\mathcal{M})$ is $\iota(\bT)$.
\end{Proposition}
\begin{proof}
The fact that $\iota(\bT)$ preserves $\mathcal{M}$ follows by Proposition \ref{toruseqprop}. So we only need to show that there is no larger subtorus that does.

Let $(b,c,\hbar) \in \bT'$ and suppose that $\Phi(\mathcal{M})$ is fixed by $(b,c,\hbar)$. Let $(X,Y,I,J)\in T^{*}\text{Rep}(\dv,\dw)$ be a representative of a point $p\in \mathcal{M}$. If $(b,c,\hbar) \cdot \Phi(p)=\Phi(p')  \in \Phi(\mathcal{M})$, then there exists a representative $(X',Y',I',J')$ of $p'$ so that 
\[
(b,c,\hbar) \cdot \Phi(X,Y,I,J) = g \cdot \Phi(X',Y',I',J')
\]
for some $g \in \prod_{i \in Q_{0}} GL(V'_i)$. Writing $g=(g_i)_{i \in Q_{0}}$ and 
\[
g_i = \begin{pmatrix}
g_{i,1} & g_{i,2} \\ g_{i,3} & g_{i,4}
\end{pmatrix}
\]
it follows by the same reasoning as in the proof of Proposition \ref{injective} that $g_{i,2}=0$ and $g_{i,3}=0$. And we have
\begin{multline*}
\begin{pmatrix}
I_n \text{diag}(b_{n,1},\ldots,b_{n,\dw_n})^{-1} & 0 & X_{n-1} \ldots X_{1} A_{1} c_n^{-1} \\
0 & \text{diag}(c_1,\ldots,c_{n-1})^{-1} & 0
\end{pmatrix} \\
=\begin{pmatrix}
g_{n,1} I_n' & 0 & g_{n,1} X_{n-1}' \ldots X_1' A_1' \\
0 & g_{n,4} \mathbb{I}_{n-1} & 0
\end{pmatrix}
\end{multline*}
Also
\begin{multline*}
\hbar^{-1} \cdot
\begin{pmatrix}
\text{diag}(b_{n,1},\ldots,b_{n,\dw_n}) J_n  & 0  \\
c_1 B_1 Y_1 \ldots Y_{n-1} & c_1 C_n \\
0 & -\text{diag}(c_2,\ldots,c_n)
\end{pmatrix} \\
=   \begin{pmatrix}
 J_n' (g_{n,1})^{-1}  & 0  \\
 B_1' Y_1' \ldots Y_{n-1}' g_{n,1}^{-1} & c_1 C_n' g_{n,4}^{-1} \\
0 & -g_{n,4}^{-1}
\end{pmatrix}
\end{multline*}
These two equations imply that
\[
\text{diag}(c_1,\ldots,c_{n-1})^{-1} = g_{n,4}= \hbar \cdot \text{diag}(c_2,\ldots,c_n)^{-1}
\]
Hence
\[
c_j=\hbar^{-1} c_{j+1} \implies c_j = \hbar^{j-n} c_n
\]
which implies that $(b,c,\hbar)$ is in $\iota(\bT)$.

\end{proof}

We also have
\begin{Corollary}
The embedding $\Phi$ maps the $\bT$-fixed locus to the $\iota(\bT)$-fixed locus.
\end{Corollary}

\section{The embedding: general case}\label{embeddings}

\subsection{One step}\label{onestepsection}
Now we apply to local construction from the previous section to the case of a general type $A$ quiver variety. So choose a natural number $m\geq 2$ and consider the $A_{m}$ quiver, i.e. the quiver with vertices $Q_{0}=\{1,2,\ldots,m\}$ and edges $i\to i+1$ for $1 \leq i \leq m-1$. Let $\dv,\dw \in \mathbb{Z}_{\geq 0}^{Q_{0}}$. Fix a stability parameter $\theta \in \mathbb{Z}^{Q_{0}}$. We will suppress $\theta$ from the notation. Consider the corresponding quiver variety $\mathcal{M}(\dv,\dw)$. Let $k< m$ be the maximal integer such that $\dw_{k+1}\neq 0$. We assume that such a $k$ exists. Let $n=m-k$. See Figure \ref{typeAquiver}.

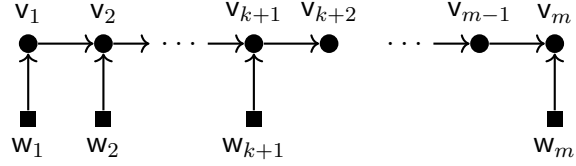
\begin{figure}[ht]
\centering
\begin{tikzpicture}[roundnode/.style={circle,fill,inner sep=2.5pt},refnode/.style={circle,inner sep=0pt},squarednode/.style={rectangle,fill,inner sep=3pt}]
\begin{scope}[shift={(0,0)}]
\node[squarednode,label=below:{$\dw_{1}$}](B0) at (-3,-1){};
\node[squarednode,label=below:{$\dw_{2}$}](B1) at (-2,-1){};
\node[squarednode,label=below:{$\dw_{k+1}$}](F0) at (0,-1){};
\node[squarednode, label=below:{$\dw_{m}$}](F2) at (4,-1){};

\node[roundnode, label=above:{$\dv_{1}$}] (A0) at (-3,0){};
\node[roundnode, label=above:{$\dv_{2}$}] (A1) at (-2,0){};
\node (L) at (-1,0){$\ldots$};
\node[roundnode, label=above:{$\dv_{k+1}$}] (V0) at (0,0){};
\node[roundnode, label=above:{$\dv_{k+2}$}](V1) at (1,0){};
\node (V2) at (2,0){$\ldots$};
\node[roundnode, label=above:{$\dv_{m-1}$}](V3) at (3,0){};
\node[roundnode, label=above:{$\dv_{m}$}](V4) at (4,0){};
\draw[thick, ->] (B0) -- (A0);
\draw[thick, ->] (B1) -- (A1);
\draw[thick, ->] (A0) -- (A1);
\draw[thick, ->] (A1) -- (L);

\draw[thick, ->] (F0) -- (V0);
\draw[thick, ->] (F2) -- (V4);
\draw[thick, ->] (L) -- (V0);
\draw[thick, ->] (V0) -- (V1);
\draw[thick, ->] (V2) -- (V3);
\draw[thick, ->] (V3) -- (V4);
\end{scope}
\end{tikzpicture}
\label{typeAquiver}
\caption{The framed quiver data for a type $A$ quiver variety. We omit drawing the framed vertices if the corresponding framing dimension is $0$.}
\end{figure}

We apply (\ref{localphi}) to the data arising from the full subquiver with vertices $\{k+1,k+2,\ldots,k+n\}$ to deduce the following.

\begin{Theorem}\label{embedthm}
   There is a torus equivariant embedding 
\begin{align}\label{onestep}
\mathcal{M}(\dv,\dw) \hookrightarrow \mathcal{M}(\dv',\dw')
\end{align}
where 
\begin{align*}
    \dv'_{i}&=\begin{cases}
        \dv_{i} & 1 \leq i \leq k \\
        \dv_{i}+i-k-1 & k+1 \leq i \leq m
    \end{cases} \\
    \dw'_{i}&=\begin{cases}
         \dw_{k+1}-1 & i=k+1 \\
        \dw_{m}+m-k & i=m \\
         \dw_{i} & \text{otherwise}
    \end{cases}
\end{align*}
\end{Theorem}
We view this as a trading a framing at vertex $k+1$ for many framings at the last vertex. 

\subsection{Repeated embeddings}\label{manystepsection}
Repeating the procedure of the previous section $\dw_{k+1}$ times, we obtain an embedding 
$$
\mathcal{M}(\dv,\dw) \hookrightarrow \mathcal{M}'
$$
where $\mathcal{M}'$ is a quiver variety with $\dw_{k+1}=0$. Continuing inductively, we obtain an embedding of $\mathcal{M}(\dv,\dw)$ into a type $A$ quiver variety with $\dw_{i}=0$ for $i \neq m$.

\begin{Theorem}\label{globalphi}
    Any type $A_{m}$ Nakajima quiver variety can be equivariantly embedded into an $A_{m}$ quiver variety with all framings at the rightmost vertex.
\end{Theorem}

Let $\theta^{\pm}=\pm(1,1,\ldots,1)$. In the special case where $\theta=\theta^{\pm} $, we obtain the following.
\begin{Corollary}
   Any type $A_{m}$ Nakajima quiver variety with stability condition $\theta^{+}$ or $\theta^{-}$ can be equivariantly embedded into the cotangent bundle of an $m$-step partial flag variety.
\end{Corollary}

\subsection{Combinatorics of torus fixed points}
Now we specialize to the case $\theta=\theta^{-}$. Fixed points on the Nakajima variety $\mathcal{M}(\dv,\dw)$ are indexed by certain $|\dw|$-tuples of partitions, see for example Proposition 4 in \cite{dinkinselliptic}.

Let $\lambda=(\lambda_1\geq \lambda_{2} \geq \ldots)$ be a partition. The length of $\lambda$ is denoted by $l(\lambda)$, and it will be useful to use the convention that $\lambda_{i}=0$ for $i > l(\lambda)$. The Young diagram of $\lambda$ is the collection of points in the plane given by $\{(i,j) \, \mid \, 1 \leq i \leq l(\lambda), 1\leq j \leq \lambda_{i}\}$. We will also refer to these points as ``boxes" in the Young diagram. If $\square=(i,j)$ is such a box, we write $\square \in \lambda$. Given $\square=(i,j) \in \lambda$, the content and height of $\square$ are defined by $\text{cont}_{\lambda}(\square)=i-j$ and $\text{ht}_{\lambda}(\square)=i+j-2$.

\begin{Definition}\label{vwpart}
A $(\mathsf{v},\mathsf{w})$-tuple of partitions is a $(\mathsf{w}_1+\ldots +\mathsf{w}_{m})$-tuple of partitions 
$$
\llambda=\left(\lambda^{i,j}\right)_{\substack{1 \leq i \leq m \\ 1 \leq j \leq \mathsf{w}_i}}
$$
such that the set
$$
S_{l}:=\bigcup_{\substack{1 \leq i \leq m \\ 1 \leq j \leq \mathsf{w}_i}}\{\square \in \lambda^{i,j} \, \mid \, \text{cont}_{\lambda^{i,j}}(\square)+i=l \}
$$
has size $\mathsf{v}_l$ for $l=1,\ldots m$.
\end{Definition}

\begin{Definition}
    Let $\square \in \llambda$; in particular, $\square \in \lambda^{i,j}$ for some $i$ and $j$. We associate an equivariant parameter of the torus $\bT$ from section \ref{torussection} to $\square$ by 
    $$
a_{\square}=a_{i,j}
    $$
\end{Definition}

\begin{Proposition}[\cite{dinkinselliptic} Proposition 4]
    $\bT$ fixed points on $\mathcal{M}(\dv,\dw)$ are in natural bijection with $(\dv,\dw)$-tuples of partitions.
\end{Proposition}

Given a $(\dv,\dw)$-tuple of partitions $\llambda$, we will write $\square \in \llambda$ to mean that $\square \in \lambda^{i,j}$ for some $i$ and $j$.

\begin{Definition}
    Let $\llambda$ be a $(\dv,\dw)$-tuple of partitions. Let $\square \in \llambda$, and in particular, suppose that $\square \in \lambda^{i,j}$. Then we define
    $$
    \gamma(\square):=\text{cont}_{\lambda^{i,j}}(\square)+i
    $$
    We also define 
    $$
    \delta^{\llambda}_{\square}=\text{ht}_{\lambda^{i,j}}(\square)
    $$
\end{Definition}

The function $\gamma$ is just a shifted version of the content. The number $\delta^{\llambda}_{\square}$ is just the height; however, we introduce this notation because in section \ref{proof} we will consider subpartitions $\llambda \subset \mmu$ and will need to distinguish between $\delta^{\llambda}_{\square}$ and $\delta^{\mmu}_{\square}$.

The vector spaces $V_{i}$ in the definition of $\mathcal{M}(\dv,\dw)$ give rise to tautological vector bundles $\mathcal{V}_{i}$ over $\mathcal{M}(\dv,\dw)$. It is known that the $\bT$ character of the fiber of these vector bundles over a fixed point $\llambda$ is given by 
\begin{equation}\label{tbwts}
\mathcal{V}_{i}|_{\llambda}=\sum_{\substack{\square \in \llambda \\ \gamma(\square)=i}} a_{\square} \hbar^{\delta^{\llambda}_{\square}}
\end{equation}

See Figure \ref{fpex0} for an example of a torus fixed point.

\begin{figure}[htbp]
    \centering
\begin{tikzpicture}[roundnode/.style={circle,fill,inner sep=2.5pt},refnode/.style={circle,inner sep=0pt},squarednode/.style={rectangle,fill,inner sep=3pt}] 
\node[roundnode,label=above:{2}] (N1) at (-2,0){};
\node[roundnode,label=above:{3}] (N2) at (-1,0){};
\node[roundnode,label=above:{4}] (N3) at (0,0) {};
\node[roundnode,label=above:{4}] (N4) at (1,0){};
\node[roundnode,label=above:{3}] (N5) at (2,0){};
\node[roundnode,label=above:{1}] (N6) at (3,0){};
\node[squarednode,label=below:{$1$}] (F1) at (0,-1){};
\node[squarednode,label=below:{$2$}] (F2) at (1,-1){};

% The partition (3,3,1)
\begin{scope}[shift={(0,0)}]
\draw[thick,-](0,0)--(-3,3)--(-1,5)--(1,3)--(2,4)--(3,3)--(0,0);
\draw[thick,-](1,1)--(-2,4);
\draw[thick,-](2,2)--(1,3);
\draw[thick,-](-1,1)--(1,3);
\draw[thick,-](-2,2)--(0,4);
\node at (0,1){$a_{3,1}$};
\node at (-1,2){$a_{3,1} \hbar$};
\node at (-2,3){$a_{3,1} \hbar^{2}$};
\node at (1,2){$a_{3,1}$};
\node at (0,3){$a_{3,1}\hbar$};
\node at (-1,4){$a_{3,1}\hbar^{2}$};
\node at (2,3){$a_{3,1}$};
\end{scope}

% The partition (2,2)
\begin{scope}[shift={(1,3.5)}]
\draw[thick,-](0,0)--(-2,2)--(0,4)--(2,2)--(0,0);
\draw[thick,-](1,1)--(-1,3);
\draw[thick,-](-1,1)--(1,3);
\node at (0,1){$a_{4,1}$};
\node at (-1,2){$a_{4,1}\hbar$};
\node at (1,2){$a_{4,1}$};
\node at (0,3){$a_{4,1}\hbar$};
\end{scope}

% The partition (4,1,1)
\begin{scope}[shift={(1,8)}]
\draw[thick,-](0,0)--(-4,4)--(-3,5)--(0,2)--(2,4)--(3,3)--(0,0);
\draw[thick,-](-1,1)--(0,2);
\draw[thick,-](-2,2)--(-1,3);
\draw[thick,-](-3,3)--(-2,4);
\draw[thick,-](1,1)--(0,2);
\draw[thick,-](2,2)--(1,3);
\node at (0,1){$a_{4,2}$};
\node at (-1,2){$a_{4,2}\hbar$};
\node at (-2,3){$a_{4,2}\hbar^{2}$};
\node at (-3,4){$a_{4,2}\hbar^{3}$};
\node at (1,2){$a_{4,2}$};
\node at (2,3){$a_{4,2}$};
\end{scope}

\draw[thick, ->] (N1) -- (N2);
\draw[thick, ->] (N2) -- (N3);
\draw[thick, ->] (N3) -- (N4);
\draw[thick, ->] (N4) -- (N5);
\draw[thick, ->] (N5) -- (N6);
\draw[thick, ->] (F1) -- (N3);
\draw[thick, ->] (F2) -- (N4);

\draw[dotted, -](0,0)--(0,12.5);
\draw[dotted, -](1,0)--(1,12.5);
\draw[dotted, -](2,0)--(2,12.5);
\draw[dotted, -](3,0)--(3,12.5);
\draw[dotted, -](-1,0)--(-1,12.5);
\draw[dotted, -](-2,0)--(-2,12.5);

\end{tikzpicture}

 \caption{An example of the torus fixed point $\llambda=((3,3,1),(2,2),(4,1,1))$ on the quiver variety $\mathcal{M}((2,3,4,4,3,1),(0,0,1,2,0,0))$. We have filled each box with $a_{\square} \hbar^{\delta^{\llambda}_{\square}}$. The $\bT$ character of the tautological bundle $\mathscr{V}_{3}$ is $a_{3,1}+a_{3,1}\hbar+a_{4,1}\hbar+a_{4,2}\hbar$.}\label{fpex0}
\end{figure}
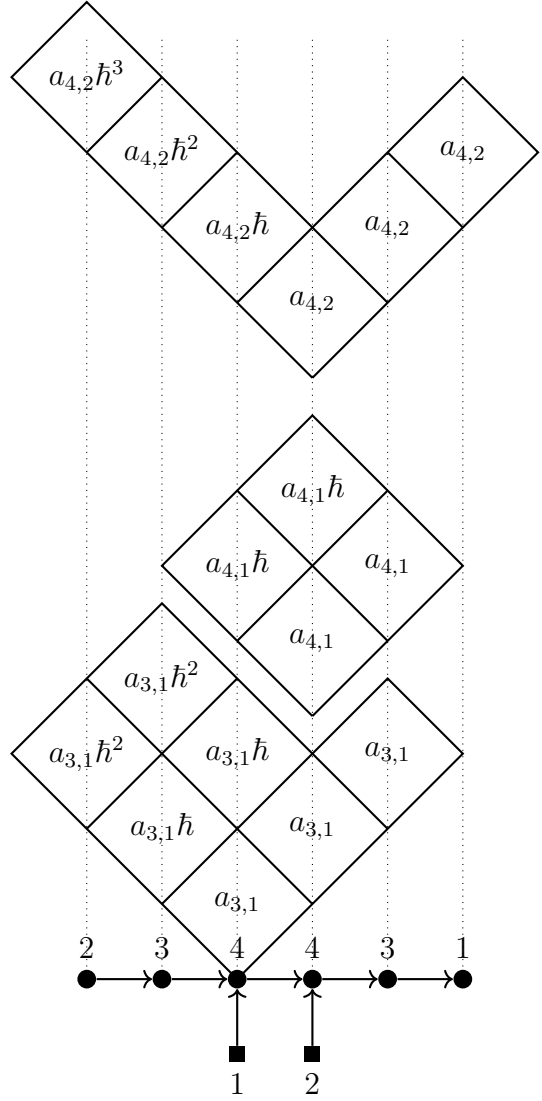

\subsection{Correspondence between torus fixed points}\label{fpmap}
Let us consider the embedding (\ref{onestep}).

Let $\llambda=\left(\lambda^{i,j} \right)_{\substack{1 \leq i \leq m\\ 1 \leq j \leq \dw_{i}}}$ be a $(\dv,\dw)$-tuple of partitions.

\begin{Definition}\label{fpmatch}
Let $\llambda$ be a $(\dv,\dw)$-tuple of partitions. We define another tuple of partitions $\mmu=\left( \mu^{i,j}\right)_{\substack{1 \leq i \leq m \\ 1 \leq j \leq \dw'_{i}}}$ by
\begin{itemize}
    \item $\mu^{i,j}= \lambda^{i,j}$ if $1 \leq i \leq k$.
\item $\mu^{k+1,j}=\lambda^{k+1,j+1}$ if $1 \leq j \leq \dw'_{k+1}=\dw_{k+1}-1$.
\item $\mu^{i,j}=\emptyset$ if $k+1< i <m$.
\item $\mu^{m,j}=\lambda^{m,j}$ for $1\leq j \leq \dw_{m}$.
\item $\mu^{m,\dw_{m}+j}=(\lambda^{k+1,1}_{j}+n-j)$ for $1 \leq j \leq n$.
\end{itemize}
\end{Definition}

 In other words, $\mmu$ is obtained from $\llambda$ by removing the partition corresponding to the first framing at vertex $k+1$ and by adding certain length $1$ partitions at the last vertex depending on the partition that was removed. 

 \begin{Proposition}
     With the notation above, $\mmu$ is a $(\dv',\dw')$-tuple of partitions.
 \end{Proposition}
 \begin{proof}
     This follows by direct computation from the definitions of $\dv'$, $\dw'$, and $\mmu$.
 \end{proof}

\begin{Proposition}\label{fpcorr}
    Let $F_{\llambda}$ be the $\iota(\bT)$ fixed component of $\mathcal{M}(\dv',\dw')$ containing $\llambda$. Then $\mmu \in F_{\llambda} \cap \mathcal{M}(\dv',\dw')^{\bT'}$.
\end{Proposition}
\begin{proof}
    This follows from the construction of $\Phi$ in section \ref{localconstruction}.
\end{proof}

 Consider the diagram
\begin{center}
\begin{tikzcd}
    K_{\bT'\times \mathbb{C}^{\times}_{q}}(\mathcal{M}_{\theta}(\dv',\dw')) \arrow[r] \arrow[d,"\Phi^{*}"] & K_{\bT'\times\mathbb{C}^{\times}_{q}}(\mu)  \arrow[d,"\iota^{*}"] \\
    K_{\bT\times\mathbb{C}^{\times}_{q}}(\mathcal{M}_{\theta}(\dv,\dw))  \arrow[r,] & K_{\bT \times\mathbb{C}^{\times}_{q}}(\lambda)
\end{tikzcd}
\end{center}
where the horizontal maps are restrictions to torus fixed points.

\begin{Lemma}\label{fpcommute}
    The previous diagram commutes.
\end{Lemma}
\begin{proof}
    Let $\mathcal{V}_{i}'$ be a tautological bundle on $\mathcal{M}(\dv',\dw')$. Formula (\ref{tbwts}) and the definition of (\ref{onestep}) show that $\iota^{*}(\mathcal{V}_{i}'|_{\mmu})=\Phi^{*}(\mathcal{V}_{i}')|_{\llambda}$. By \cite{kirv}, the $K$-theory of Nakajima varieties is generated by tautological classes. The lemma follows.
\end{proof}

\subsection{Examples}

To help the read parse the constructions and notation above, we provide a few examples.

\begin{Example}
Let $\dv=(1,1)$ and $\dw=(1,1)$. Then $\mathcal{M}'=\mathcal{M}((1,2),(0,3))$. The map on fixed points from Definition \ref{fpmatch} is as follows:
\begin{align*}
    ((1),(1)) \mapsto ((1),(2),(0)) \\
    ((1,1),(0)) \mapsto ((0),(2),(1)) \\
    ((0),(2)) \mapsto ((2),(1),(0))
\end{align*}
As in section \ref{torusmap}, the tori are $\bT=\mathbb{C}^{\times}_{a_{1,1}} \times \mathbb{C}^{\times}_{a_{2,1}} \times \mathbb{C}^{\times}_{\hbar}$ and $\bT'= \mathbb{C}^{\times}_{b_{2,1}}\times \mathbb{C}^{\times}_{c_{1}} \times \mathbb{C}^{\times}_{c_{2}} \times \mathbb{C}^{\times}_{\hbar}$ and the map (\ref{torimap}) is given by
$$
(a_{1,1},a_{2,1},\hbar)\mapsto (a_{2,1},\hbar^{-1} a_{1,1},a_{1,1},\hbar)
$$
\end{Example}

Let us next consider an example where we apply the construction (\ref{onestep}) several times, as in Theorem \ref{globalphi}.
\begin{Example}
    Let $\dv=(2,3,4,4,3,1)$ and $\dw=(0,0,1,2,0,0)$. The map in Theorem \ref{globalphi} is
    \begin{align}
\mathcal{M}(\dv,\dw) &\hookrightarrow \mathcal{M}((2,3,4,4,4,3),(0,0,1,1,0,3)) \nonumber\\
&\hookrightarrow \mathcal{M}((2,3,4,4,5,5)),(0,0,1,0,0,6))) \nonumber\\
&\hookrightarrow \mathcal{M}((2,3,4,5,7,8),(0,0,0,0,0,10)) \label{ex1}
\end{align}
 Consider the $(\dv,\dw)$-tuple of partitions 
    $$
    \llambda=(((3,3,2),(2,2),(4,1,1))
    $$
    Under the first step of the embedding, Definition \ref{fpmatch} sends $\llambda$ to the fixed point $((3,3,1),(2,2),(6),(2),(1))$ of $\mathcal{M}((2,3,4,4,4,3),(0,0,1,1,0,3))$. This is depicted in Figure \ref{fpex}.

    The reader can check that under the composition of all the maps in (\ref{ex1}), Definition \ref{fpmatch} maps $\llambda$ to the fixed point $((6),(2),(1),(4),(3),\emptyset,(6),(5),(2),\emptyset)$ of $\mathcal{M}((2,3,4,5,7,8),(0,0,0,0,0,10))$.
    
    As the picture shows, essentially what happens is that each column lengthens down and to the right as much as it needs to in order to be a partition with corner box above the last vertex.

The tori acting on the varieties in the first and last steps of (\ref{ex1}) are $\bT = \mathbb{C}^{\times}_{a_{3,1}} \times \mathbb{C}^{\times}_{a_{4,1}} \times \mathbb{C}^{\times}_{a_{4,2}} \times \mathbb{C}^{\times}_{\hbar}$ and $\bT'= (\mathbb{C}^{\times})^{10}\times \mathbb{C}^{\times}_{\hbar}$, respectively. The map on tori (\ref{torimap}) is
\begin{multline*}
    (a_{3,1},a_{4,1},a_{4,2},\hbar) \mapsto \\
    (\hbar^{-2} a_{4,2},\hbar^{-1} a_{4,2},a_{4,2},\hbar^{-2} a_{4,1},\hbar^{-1} a_{4,1},a_{4,1},\hbar^{-3} a_{3,1},\hbar^{-2} a_{3,1},\hbar^{-1} a_{3,1},a_{3,1},\hbar)
\end{multline*}

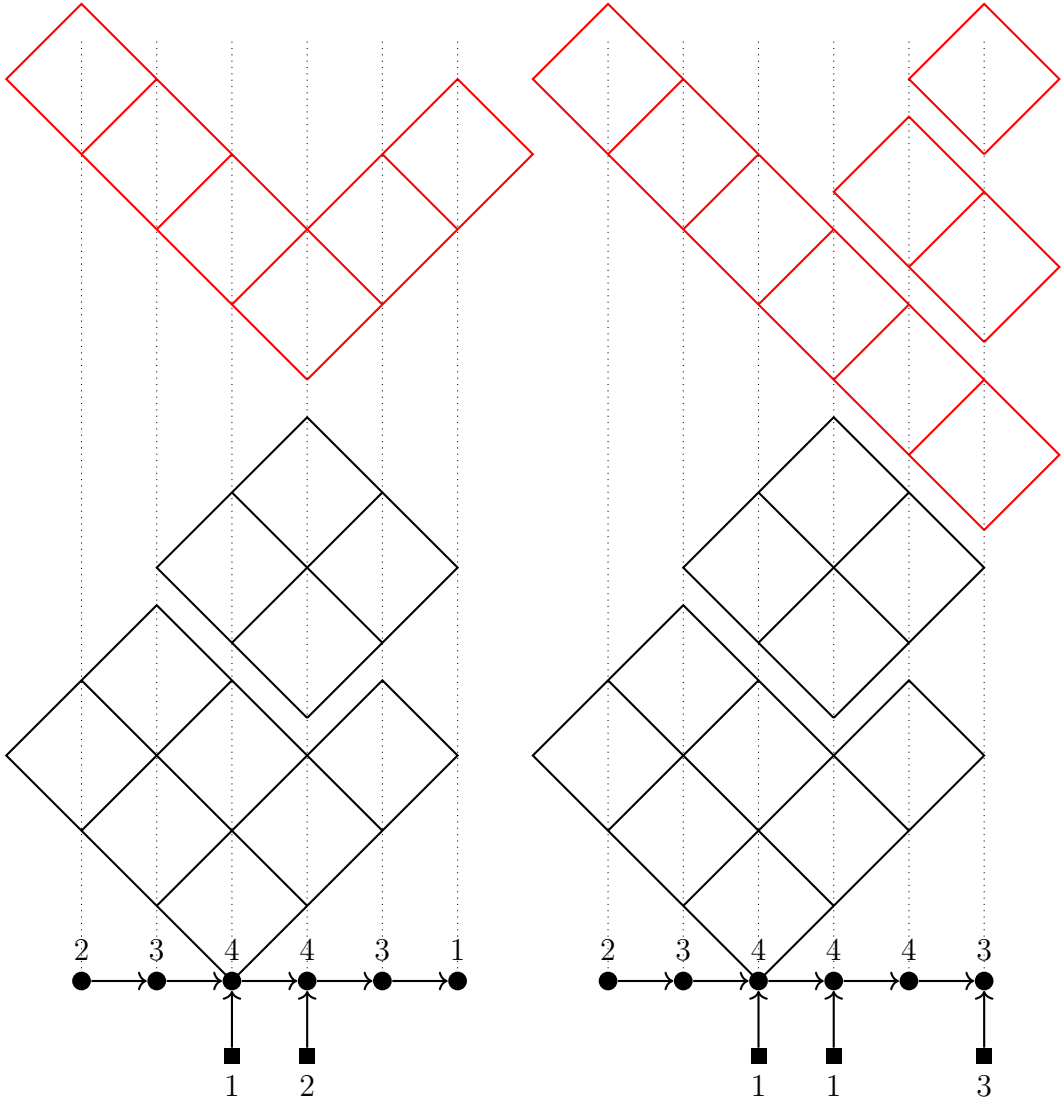
\begin{figure}[htbp]
    \centering
\begin{tikzpicture}[roundnode/.style={circle,fill,inner sep=2.5pt},refnode/.style={circle,inner sep=0pt},squarednode/.style={rectangle,fill,inner sep=3pt}] 
\node[roundnode,label=above:{2}] (N1) at (-2,0){};
\node[roundnode,label=above:{3}] (N2) at (-1,0){};
\node[roundnode,label=above:{4}] (N3) at (0,0) {};
\node[roundnode,label=above:{4}] (N4) at (1,0){};
\node[roundnode,label=above:{3}] (N5) at (2,0){};
\node[roundnode,label=above:{1}] (N6) at (3,0){};
\node[squarednode,label=below:{$1$}] (F1) at (0,-1){};
\node[squarednode,label=below:{$2$}] (F2) at (1,-1){};

% The partition (3,3,1)
\begin{scope}[shift={(0,0)}]
\draw[thick,-](0,0)--(-3,3)--(-1,5)--(1,3)--(2,4)--(3,3)--(0,0);
\draw[thick,-](1,1)--(-2,4);
\draw[thick,-](2,2)--(1,3);
\draw[thick,-](-1,1)--(1,3);
\draw[thick,-](-2,2)--(0,4);
\end{scope}

% The partition (2,2)
\begin{scope}[shift={(1,3.5)}]
\draw[thick,-](0,0)--(-2,2)--(0,4)--(2,2)--(0,0);
\draw[thick,-](1,1)--(-1,3);
\draw[thick,-](-1,1)--(1,3);
\end{scope}

% The partition (4,1,1)
\begin{scope}[shift={(1,8)},color=red]
\draw[thick,-](0,0)--(-4,4)--(-3,5)--(0,2)--(2,4)--(3,3)--(0,0);
\draw[thick,-](-1,1)--(0,2);
\draw[thick,-](-2,2)--(-1,3);
\draw[thick,-](-3,3)--(-2,4);
\draw[thick,-](1,1)--(0,2);
\draw[thick,-](2,2)--(1,3);
\end{scope}

\draw[thick, ->] (N1) -- (N2);
\draw[thick, ->] (N2) -- (N3);
\draw[thick, ->] (N3) -- (N4);
\draw[thick, ->] (N4) -- (N5);
\draw[thick, ->] (N5) -- (N6);
\draw[thick, ->] (F1) -- (N3);
\draw[thick, ->] (F2) -- (N4);

\draw[dotted, -](0,0)--(0,12.5);
\draw[dotted, -](1,0)--(1,12.5);
\draw[dotted, -](2,0)--(2,12.5);
\draw[dotted, -](3,0)--(3,12.5);
\draw[dotted, -](-1,0)--(-1,12.5);
\draw[dotted, -](-2,0)--(-2,12.5);

\begin{scope}[shift={(7,0)}]
\node[roundnode,label=above:{2}] (N1) at (-2,0){};
\node[roundnode,label=above:{3}] (N2) at (-1,0){};
\node[roundnode,label=above:{4}] (N3) at (0,0) {};
\node[roundnode,label=above:{4}] (N4) at (1,0){};
\node[roundnode,label=above:{4}] (N5) at (2,0){};
\node[roundnode,label=above:{3}] (N6) at (3,0){};
\node[squarednode,label=below:{$1$}] (F1) at (0,-1){};
\node[squarednode,label=below:{$1$}] (F2) at (1,-1){};
\node[squarednode,label=below:{$3$}] (F3) at (3,-1){};

% The partition (3,3,1)
\begin{scope}[shift={(0,0)}]
\draw[thick,-](0,0)--(-3,3)--(-1,5)--(1,3)--(2,4)--(3,3)--(0,0);
\draw[thick,-](1,1)--(-2,4);
\draw[thick,-](2,2)--(1,3);
\draw[thick,-](-1,1)--(1,3);
\draw[thick,-](-2,2)--(0,4);
\end{scope}

% The partition (2,2)
\begin{scope}[shift={(1,3.5)}]
\draw[thick,-](0,0)--(-2,2)--(0,4)--(2,2)--(0,0);
\draw[thick,-](1,1)--(-1,3);
\draw[thick,-](-1,1)--(1,3);
\end{scope}

% The partition (6)
\begin{scope}[shift={(3,6)},color=red]
\draw[thick,-](0,0)--(-6,6)--(-5,7)--(1,1)--(0,0);
\draw[thick,-](-1,1)--(0,2);
\draw[thick,-](-2,2)--(-1,3);
\draw[thick,-](-3,3)--(-2,4);
\draw[thick,-](-4,4)--(-3,5);
\draw[thick,-](-5,5)--(-4,6);
\end{scope}

% The partition (2)
\begin{scope}[shift={(3,8.5)},color=red]
\draw[thick,-](0,0)--(-2,2)--(-1,3)--(1,1)--(0,0);
\draw[thick,-](-1,1)--(0,2);
\end{scope}

% The partition (1)
\begin{scope}[shift={(3,11)},color=red]
\draw[thick,-](0,0)--(-1,1)--(0,2)--(1,1)--(0,0);
\end{scope}

\draw[thick, ->] (N1) -- (N2);
\draw[thick, ->] (N2) -- (N3);
\draw[thick, ->] (N3) -- (N4);
\draw[thick, ->] (N4) -- (N5);
\draw[thick, ->] (N5) -- (N6);
\draw[thick, ->] (F1) -- (N3);
\draw[thick, ->] (F2) -- (N4);
\draw[thick, ->] (F3) -- (N6);

\draw[dotted, -](0,0)--(0,12.5);
\draw[dotted, -](1,0)--(1,12.5);
\draw[dotted, -](2,0)--(2,12.5);
\draw[dotted, -](3,0)--(3,12.5);
\draw[dotted, -](-1,0)--(-1,12.5);
\draw[dotted, -](-2,0)--(-2,12.5);
\end{scope}

\end{tikzpicture}

 \caption{The fixed point on the left maps to the fixed point on the right.}\label{fpex}
\end{figure}

\end{Example}

\section{Quasimaps}\label{secqm}
We give a brief review of vertex function for Nakajima varieties. The foundational treatment of quasimap counting was given in \cite{qm}, and the case of quasimaps to Nakajima quiver varieties is treated in \cite{pcmilect}. It is also reviewed in \cite{dinkinsthesis} section 2.2 and \cite{dinksmir3} section 2.3.

\subsection{Localization formula}
Let $\mathcal{M}=\mathcal{M}_{-\theta^{+}}(\dv,\dw)$ be a type $A_{m}$ Nakajima quiver variety. The vertex function of $\mathcal{M}$ is an equivariant count of quasimaps from $\mathbb{P}^{1}$ to $\mathcal{M}$ nonsingular at $\infty$, see \cite{pcmilect} section 7.2 for precise definitions. See also \cite{dinkinsthesis} section 2, \cite{dinksmir3} section 2, and \cite{Pushk1} section 2. For explicit computations of vertex functions, see \cite{dinkinsthesis} section 5.1.3 and \cite{Pushk1} section 2.

A stable quasimap from $\mathbb{P}^{1}$ to $\mathcal{M}$ provides the data of 
\begin{itemize}
    \item Vector bundles $\mathscr{V}_{i}$ over $\mathbb{P}^{1}$ for $i \in Q_{0}$.
    \item Topologically trivial vector bundles $\mathscr{W}_{i}$ over $\mathbb{P}^{1}$ for $i \in Q_{0}$.
    \item A global section $s \in H^{0}\left(\mathbb{P}^{1}, \mathscr{M} \oplus \hbar^{-1} \mathscr{M}^{\vee} \right)$, where
    $$
\mathscr{M}=\bigoplus_{i \to j} \text{Hom}(\mathscr{V}_{i},\mathscr{V}_{j}) \oplus \bigoplus_{i \in Q_{0}} \text{Hom}(\mathscr{W}_{i},\mathscr{V}_{i})
    $$
\end{itemize}
such that $s(p)$ satisfies the moment map equations for all $p \in \mathbb{P}^{1}$ and $s(p)$ is a $\theta$-semistable point for all but finitely many $p \in \mathbb{P}^{1}$. Points $p$ where $s(p)$ is $\theta$-semistable are called nonsingular points of the quasimap.

The vector $d=(d_{i})_{i \in Q_{0}}=(\deg \mathscr{V}_{i})_{i \in Q_{0}} \in \mathbb{Z}^{Q_{0}}$ is called the degree of a quasimap. Let $\qm$ be the moduli space of stable quasimaps from $\mathbb{P}^{1}$ to $\mathcal{M}$. For $d \in \mathbb{Z}^{Q_{0}}$, let $\qm^{d}\subset \qm$ be the moduli space of degree $d$ quasimaps. It is known that $\qm^{d}$ admits a canonical deformation-obstruction theory, which gives rise to a virtual structure sheaf $\mathcal{O}_{\text{vir}}^{d}$ \cite{qm}. Let $T^{1/2}\mathcal{M}\in K_{\bT}(\mathcal{M})$ be the polarization of the tangent bundle of $\mathcal{M}$ given by
$$
T^{1/2}\mathcal{M}=\sum_{i \to j} \text{Hom}(\mathcal{V}_{i},\mathcal{V}_{j}) + \sum_{i \in Q_{0}} \text{Hom}(\mathcal{W}_{i},\mathcal{V}_{i}) - \sum_{i \in Q_{0}} \text{Hom}(\mathcal{V}_{i},\mathcal{V}_{i})
$$
This induces a virtual bundle $\mathscr{T}^{1/2}$ on $\mathbb{P}^{1} \times \qm$. As discussed in \cite{pcmilect} and \cite{NO}, it is better to study the symmetrized virtual structure sheaf, which is defined by 
$$
\vrs^{d} =\mathcal{O}^{d}_{\text{vir}}\otimes \left(\mathscr{K} \otimes \frac{\det \mathscr{T}^{1/2}|_{0}}{\det \mathscr{T}^{1/2}|_{\infty}}\right)^{1/2}
$$
where $\mathscr{K}_{\text{vir}}=(\det T_{\text{vir}}\qm^{d})^{-1}$.

Let $\mathbb{C}^{\times}_{q}$ act on $\mathbb{P}^{1}$ by
$$
q \cdot [x_{0}:x_{1}] = [x_{0}:q x_{1}]
$$
This induces an action of $\mathbb{C}^{\times}_{q}$ on $\qm$. The action of the torus $\bT$ on $\mathcal{M}$ also induces an action on $\qm$. We will work equivariantly with respect to $\bT_{q}=\bT \times \mathbb{C}^{\times}_{q}$. Denote $0=[0:1]$ and $\infty=[1:0]$.

Consider the open locus $\qm^{d}_{\ns \infty} \subset \qm^{d}$ of quasimaps that are nonsingular at $\infty \in \mathbb{P}^{1}$. Thus there is a well-defined evaluation map $\text{ev}_{\infty}: \qm^{d}_{\ns \infty} \to \mathcal{M}$.

The vertex function of $\mathcal{M}$ is the formal series with coefficients in the localized $K$-theory of $\mathcal{M}$ defined by:
$$
V(z)=\sum_{d} \text{ev}_{\infty,*}\left( \vrs^{d} \right) z^{d}  \in K_{\bT\times\mathbb{C}^{\times}_{q}}(\mathcal{M})_{loc}[[z]]
$$

% Here, $\text{ev}_{\infty}: \qm^{d}_{\ns \infty} \to \mathcal{M}$ is the natural evaluation map on the moduli space of degree $d$ quasimaps from $\mathbb{P}^{1}$ to $\mathcal{M}$ that are nonsingular at $\infty$. 
% $$
% \vrs^{d}=\mathcal{O}_{\text{vir}}^{d} \otimes \left(\mathscr{K}_{\text{vir}} \otimes\frac{\det \mathscr{T}^{1/2}|_{\infty}}{\det \mathscr{T}^{1/2}|_{0}}\right)^{1/2}
% $$
% where $\mathscr{T}^{1/2}$ is the virtual bundle on $\mathbb{P}^{1} \times \qm^{d}_{\ns \infty}$ induced by a chosen polarization $T^{1/2}$ of the tangent bundle of $\mathcal{M}$ and $\mathscr{K}_{\text{vir}}=(\det T_{\text{vir}}\qm_{\ns \infty})^{-1}$.

Localization is necessary to define the vertex function since $\text{ev}_{\infty}$ is not a proper map. The set of degrees of all quasimaps forms a cone in $\mathbb{Z}^{Q_{0}}$. Given $d\in \mathbb{Z}^{Q_{0}}$, we understand $z^{d}$ to be an element of the semigroup algebra of this cone. The notation $[[z]]$ above stands for formal series in $z^{d}$ as $d$ ranges over all possible quasimap degrees. 

The vertex function of $\mathcal{M}$ can be computed by equivariant localization with respect to $\bT_{q}$ in the following way. Define the function $\ahat$ on weights of $\bT_{q}$ by $\ahat(x)=\frac{1}{x^{1/2}-x^{-1/2}}$, and extend it by multiplicativity to sums and differences of weights. This should be thought of as a symmetrized version of the function $x\mapsto \frac{1}{1-x^{-1}}$ that appears in the usual $K$-theoretic equivariant localization formula, see Remark \ref{remahat} below.

By localization, the restriction of the vertex function to $p \in \mathcal{M}^{\bT}$ can be calculated as
\begin{equation}\label{verdef}
V(z)|_{p} = \sum_{\substack{f \in (\qm_{\ns \infty})^{\bT_{q}} \\ f(\infty)=p}} z^{\deg f} \ahat\left(T_{\text{vir},f} \qm - T_{p}\mathcal{M} \right) q^{\deg \mathscr{T}^{1/2}/2}
\end{equation}
In this formula, $T_{p}\mathcal{M}$ and the virtual tangent space $T_{\text{vir},f} \qm$ are identified with their $\bT_{q}$ characters, and hence lie in the domain of $\ahat$.

\begin{Remark}\label{remahat}
    Tangent weights contribute to the vertex function via $\ahat$ rather than the usual $K$-theoretic Euler class because the vertex function is defined using the \textit{symmetrized} virtual structure sheaf, see \cite{pcmilect} section 6.1. This is also the reason for the appearance of $q^{\deg \mathscr{T}^{1/2}/2}$.
\end{Remark}
\begin{Remark}
    The subtraction of $T_{p}\mathcal{M}$ is simply a normalization condition and has the effect of making the vertex function start with $1$.
\end{Remark}

\subsection{Explicit formula for vertex functions}
By definition, a stable quasimap $f$ to $\mathcal{M}$ provides vector bundles $\mathscr{V}_{i}$ and $\mathscr{W}_{i}$ over $\mathbb{P}^{1}$. Let
$$
\mathcal{T}^{1/2}= \sum_{i=1}^{m-1} \text{Hom}(\mathscr{V}_{i},\mathscr{V}_{i+1}) + \sum_{i=1}^{m} \text{Hom}(\mathscr{W}_{i},\mathscr{V}_{i}) - \sum_{i=1}^{m} \text{Hom}(\mathscr{V}_{i}, \mathscr{V}_{i})
$$
The virtual tangent space at $f$ of the moduli space of quasimaps is
\begin{equation}\label{vtan}
T_{\text{vir},f} \qm = H^{*}\left(\mathbb{P}^{1}, \mathcal{T}^{1/2} + \hbar^{-1} \left(\mathcal{T}^{1/2}\right)^{\vee} \right),
\end{equation}
see \cite{pcmilect}.
This is an equality of representations of $\bT_{q}$. We compute the character of this representation.

As $\mathcal{V}_{i}$ is a vector bundle over $\mathbb{P}^{1}$, there is a decomposition 
$$
\mathcal{V}_{i}\cong \bigoplus_{j=1}^{\dv_{i}} x_{i,j}(p) \mathcal{O}(d_{i,j})
$$
where $d_{i,j} \in \mathbb{Z}$ and $\{x_{i,j}(p)\}_{1 \leq j \leq \dv_{i}}$ is the set of $\bT$-weights of the tautological bundle $\mathcal{V}_{i}$ on $\mathcal{M}$. The lines bundles $\mathcal{O}(d)$ are linearized so that the $\mathbb{C}^{\times}_{q}$ character of $H^{0}(\mathbb{P}^{1},\mathcal{O}(d))$ is $\sum_{j=0}^{d} q^{j}$ when $d\geq0$. The integers $d_{i,j}$ that appear must satisfy certain linear inequalities to arise from a quasimap, see \cite{KorZeit} section 3. We denote this set by $C_{p} \subset \mathbb{Z}^{|\dv|}$. In fact, one can show that for each $\{d_{i,j}\} \in C_{p}$, there exists a unique $\bT\times\mathbb{C}^{\times}_{q}$ fixed quasimap in $\qm_{\ns \infty}$, see \cite{dinkinsthesis} section 5.1.3.

\begin{Lemma}[\cite{Pushk1} Lemma 1]
For any weight $w$ of $\bT_{q}$, we have
\begin{multline*}
H^{*}\left(w \mathcal{O}(d)+\frac{1}{\hbar w} \mathcal{O}(-d) \right)-w - \frac{1}{\hbar w} \\ =
\begin{cases}
    w q (1+q+\ldots q^{d-1}) - \frac{1}{\hbar w}(1 + q^{-1} + \ldots + q^{-(d-1)}) & d \geq 0 \\
    \frac{q}{\hbar w}(1 + q + \ldots q^{-d-1}) - w(1+ q^{-1} +\ldots q^{d+1}) & d < 0
\end{cases} \\
=\begin{cases}
    w q \sum_{i=0}^{d-1} q^{i} - \frac{1}{\hbar w}\sum_{i=-(d-1)}^{0} q^i & d \geq 0 \\
    \frac{q}{\hbar w}\sum_{i=0}^{-d-1} q^i - w\sum_{i=d+1}^{0} q^i & d < 0
\end{cases}
\end{multline*}
\end{Lemma}
\begin{proof}
    The equivariant dualizing sheaf on $\mathbb{P}^1$ is $q \mathcal{O}(-2)$, where the $q$ means to twist $\mathcal{O}(-2)$ by the trivial line bundle with equivariant weight $q$. The result follows by Serre duality.
\end{proof}

\begin{Lemma}\label{qpoch}
    $$
\ahat\left( H^{*}\left(w \mathcal{O}(d)+\frac{1}{\hbar w} \mathcal{O}(-d) \right)-w - \frac{1}{\hbar w}  \right)=\frac{(\hbar w)_{d}}{(q w)_{d}}\left(-\frac{q^{1/2}}{\hbar^{1/2}} \right)^{d}
    $$
    where $(x)_{d}:=\frac{(x)_{\infty}}{(x q^{d})_{\infty}}$ and $(x)_{\infty}=\prod_{i=0}^{\infty}(1-x q^{i})$.
\end{Lemma}
\begin{proof}
    This follows by applying the function $\ahat$ to the previous lemma.
\end{proof}

Along with the factor of $q^{\deg \mathscr{T}^{1/2}/2}$, the contribution of $w \mathcal{O}(d)+\frac{1}{\hbar w} O(-d)$ to the vertex function is thus
$$
\frac{(\hbar w)_{d}}{(q w)_{d}}\left(-\frac{q}{\hbar^{1/2}} \right)^{d}
$$

Denote 
$$
\{x\}_{d}=\frac{(\hbar x)_{d}}{(q x)_{d}}\left(-\frac{q}{\hbar^{1/2}}\right)^{d}
$$

Combining Lemma \ref{qpoch}, (\ref{verdef}), and (\ref{vtan}), we deduce the following.

\begin{Theorem}\label{vertex}
The restriction of the vertex function to $p$ is
    \begin{multline*}
        V(z)|_{p}= \sum_{\{d_{i,j}\} \in C_{p}} z^{d} \prod_{i=1}^{m-1} \prod_{j=1}^{\dv_i} \prod_{k=1}^{\dv_{i+1}} \left\{\frac{x_{i+1,k}(p)}{x_{i,j}(p)} \right\}_{d_{i+1,k}-d_{i,j}} \\
        \prod_{i=1}^{m} \prod_{j=1}^{\dw_{i}} \prod_{k=1}^{\dv_{i}} \left\{ \frac{x_{i,k}(p)}{a_{i,j}}\right\}_{d_{i,k}} \prod_{i=1}^{m} \prod_{j=1}^{\dv_{i}} \prod_{k=1}^{\dv_{i}} \left\{ \frac{x_{i,k}(p)}{x_{i,j}(p)}\right\}_{d_{i,k}-d_{i,j}}
    \end{multline*}
    where $z^{d}=\prod_{i=1}^{m}\prod_{j=1}^{\dv_{i}} z_{i}^{d_{i,j}}$.
\end{Theorem}

\begin{Remark}
    As already mentioned, the set of degrees $\{d_{i,j}\}$ must satisfy certain conditions to arise from a quasimap. However, the formula in Theorem \ref{vertex} turns out to be quite miraculous: it suffices to take $d_{i,j}\geq 0$ for all $i$ and $j$ and all terms not arising from quasimaps are automatically zero, see \cite{dinksmir2} Proposition 7. 
\end{Remark}

\begin{Remark}
    When $\mathcal{M}$ is the cotangent bundle to $\mathbb{P}^{n}$, Theorem \ref{vertex} shows that the vertex function is equal to the usual basic $q$-hypergeometric function $_{n+1} \phi_{n}$ for certain values of the parameters, see \cite{hypergeo}.
\end{Remark}

Let $\llambda \in \mathcal{M}(\dv,\dw)^{\bT}$. Recall the weights of the tautological bundle $\mathcal{V}_{i}$ from (\ref{tbwts}). Thus we can write the vertex function as
\begin{multline*}
  V(z)|_{\llambda}= \sum_{\{d_{\square}\} \in C_{\llambda}} z^{d} \prod_{\substack{\square,\square' \in \llambda \\ \gamma(\square')=\gamma(\square)+1}} \left\{\frac{a_{\square'}}{a_{\square}}\hbar^{\delta^{\llambda}_{\square'}-\delta^{\llambda}_{\square}} \right\}_{d_{\square'}-d_{\square}} \\
  \prod_{\square \in \llambda} \prod_{j=1}^{\dw_{\gamma(\square)}} \left\{\frac{a_{\square}}{a_{\gamma(\square),j}} \hbar^{\delta^{\llambda}_{\square}} \right\}_{d_{\square}} \prod_{\substack{\square,\square'\in \llambda \\ \gamma(\square)=\gamma(\square')}} \left\{\frac{a_{\square'}}{a_{\square}} \hbar^{\delta^{\llambda}_{\square'}-\delta^{\llambda}_{\square}} \right\}_{d_{\square'}-d_{\square}}^{-1}
\end{multline*}

\begin{Proposition}[\cite{dinksmir2} Proposition 7]\label{degrees}
A collection $\{d_{\square}\}_{\square \in \llambda}$ lies in $C_{\llambda}$ if and only if the following two conditions hold:
\begin{enumerate}
    \item $d_{\square} \geq 0$ for all $\square \in \llambda$
    \item If $\square,\square' \in \lambda^{i,j}$ satisfy $\gamma(\square')=\gamma(\square)\pm 1$ and $\delta^{\llambda}_{\square'}=\delta^{\llambda}_{\square}+1$, then $d_{\square} \leq d_{\square'}$.
\end{enumerate}
\end{Proposition}

\begin{Definition}
    We call $C_{\llambda}$ the set of admissible degrees for $\llambda$.
\end{Definition}

\subsection{Preservation of vertex functions}

Recall the notation of section \ref{onestepsection}. We have $\Phi:\mathcal{M}(\dv,\dw)\hookrightarrow\mathcal{M}'(\dv',\dw')$ as in (\ref{onestep}). These are quiver varieties from a type $A$ quiver with $m$ vertices. Recall the definitions of $k$ and $n$ used there. There are tori $\bT$ and $\bT'$ that act on these two varieties and an embedding $\iota: \bT \to \bT'$. The map $\Phi$ is equivariant with respect to $\iota$ and hence induces a pullback
$$
\Phi^{*}: K_{\bT'\times\mathbb{C}^{\times}_{q}}(\mathcal{M}')_{loc} \to K_{\bT\times\mathbb{C}^{\times}_{q}}(\mathcal{M})_{loc}
$$

Our main theorem is the following.
\begin{Theorem}\label{mainthm}
Up to shifts of $z$, $\Phi^{*}$ preserves vertex functions. More specifically,
    $$
\Phi^{*}(V'(z))=V(\tilde{z})
    $$
    where $\tilde{z}$ stands for the shift $z_{j}\mapsto z_{j} q^{-1}$ for $k+1 \leq j <m$ and $z_{m} \mapsto z_{m} q^{n-1}$.
\end{Theorem}

Let 
$$
\iota^{*}: K_{\bT' \times \mathbb{C}^{\times}_{q}}(pt) \to K_{\bT\times\mathbb{C}^{\times}_{q}}(pt)
$$
be the induced map. Theorem \ref{mainthm} is equivalent to the following.

Let $\llambda \in \mathcal{M}^{\bT}$. Let $\mmu\in\mathcal{M}'^{\bT'}$ be the fixed point constructed in Definition \ref{fpmatch}.

In view of Lemma \ref{fpcommute}, the Theorem \ref{mainthm} is equivalent to the following.

\begin{Theorem}\label{mainthmloc}
    Let $\llambda \in \mathcal{M}^{\bT}$ and let $\mmu\in\mathcal{M}'^{\bT'}$ be the fixed point constructed in section \ref{fpmap}. Then
    $$
\iota^{*}(V'(z)|_{\mmu})=V(\tilde{z})|_{\llambda}
    $$
\end{Theorem}

Repeating the embedding procedure as described in \ref{manystepsection}, we obtain the following.

\begin{Corollary}
    Up to shifts of $z_{1},\ldots,z_{m}$ by powers of $q$, the vertex function of any type $A$ quiver variety can be obtained by a certain specialization of the equivariant parameters of the vertex function of the cotangent bundle of a partial flag variety.
\end{Corollary}

We will prove Theorem \ref{mainthm} in section \ref{proof} by a careful analysis of the localization formula in Theorem \ref{vertex}.

\section{Proof of Theorem \ref{mainthmloc}}\label{proof}

\subsection{Setup}

We review the setup of section \ref{onestep}. Let $\dv,\dw  \in \mathbb{Z}^{m}_{\geq 0}$ and let $k<m$ be the maximal index such that $\dw_{k+1}\neq 0$. Define $n$ by $k+n=m$. Then
\begin{align*}
\dw &=(\dw_{1},\dw_{2},\ldots,\dw_{k},\dw_{k+1},0,\ldots,0,\dw_{k+n})\\
\dv'&=(\dv_{1},\dv_{2},\ldots,\dv_{k+1},\dv_{k+2}+1,\ldots,\dv_{k+n}+n-1) \\
\dw' &=(\dw_{1},\dw_{2},\ldots,\dw_{k+1}-1,0,\ldots,0,\dw_{k+n}+n)
\end{align*}
Let $\mathcal{M}:=\mathcal{M}_{-\theta^{+}}(\dv,\dw)$ and $\mathcal{M}':=\mathcal{M}_{-\theta^{+}}(\dv',\dw')$ and consider the embedding $\Phi: \mathcal{M} \hookrightarrow \mathcal{M}'$ and the embedding $\iota: \bT \to \bT'$. Let $\iota^*$ be the induced pullback on torus weights.

Denote the equivariant parameters of $\bT$ by $a_{i,j}$ where $1 \leq i \leq m$ and $1 \leq j \leq \dw_{i}$ and the equivariant parameters of $\bT'$ by $b_{i,j}$ where $1 \leq i \leq m$ and $1 \leq j \leq \dw'_{i}$. We will also denote $c_{j}:=b_{m,\dw_{m}+j}$ for $1 \leq j \leq n$. By (\ref{torimap}), the map $\iota^*$ is defined by 
\begin{align*}
    \iota^*(b_{i,j})=\begin{cases}
a_{i,j} & 1 \leq i \leq k \\
a_{k+1,j+1} & i=k+1 \text{ and } 1\leq j \leq \dw_{k+1}-1 \\
a_{m,j} & i=m \text{ and } 1 \leq j \leq \dw_{m}\\
\hbar^{j-\dw_{m}-n} a_{k+1,1} & i=m \text{ and }  \dw_{m}+1 \leq j \leq \dw_{m}+n
    \end{cases}
\end{align*}

We abbreviate $a:=a_{k+1,1}$ so that the last line reads $\iota^{*}(c_{j})=\hbar^{j-n} a$ for $1\leq j \leq n$.

Let $\llambda$ be a $(\dv,\dw)$-tuple of partitions and let $\mmu$ be as in Definition \ref{fpmatch}. To simplify notations below, we will denote $\lambda:=\lambda^{k+1,1}$ and 
$$
\nu:=(\nu^{1},\nu^{2},\ldots,\nu^{n}):=(\mu^{m,\dw_{m}+1},\ldots,\mu^{m,\dw_{m}+n})
$$
Since $\mu^{m,\dw_{m}+j}=(\lambda_{j}+n-j)$, we will canonically identify the Young diagram of $\lambda$ with a subset of the disjoint union of the Young diagrams of each $\mu^{m,\dw_{m}+j}$ for $1 \leq j \leq n$. By section \ref{fpmap}, $\mu^{i,j}=\lambda^{i,j}$ if $i \neq m$ or if $i=m$ and $j\leq \dw_{m}$. We identify such boxes $\square \in \llambda$ with their counterparts in $\mmu$. So, we have defined an inclusion $\llambda \subset \mmu$; for example, see Figure \ref{boxdegs}.  

\subsection{Comparison of degrees}

We denote by $V'_{\mmu}$ the restriction of the vertex function of $\mathcal{M}'$ to $\mmu$. We similarly write $V_{\llambda}$ for the restriction of the vertex function of $\mathcal{M}$ to $\llambda$. We will denote the coefficients of the vertex function $V'_{\mmu}$ by $C^{\mmu}_{\{d_{\square}\}}(b,\hbar)$, so that
$$
V'_{\mmu}=\sum_{\{d_{\square}\} \in C_{\mmu}} z^{d} C^{\mmu}_{\{d_{\square}\}}(b,\hbar)
$$
and similarly for $C^{\llambda}_{\{d_{\square}\}}(a,\hbar)$

Let $\{d_{\square}\}_{\square \in \mmu}$ be an admissible collection of degrees appearing in the vertex function $V'_{\mmu}$. By forgetting the degrees corresponding to boxes $\square \in \mmu \setminus \llambda$, we obtain a set $\{d_{\square}\}_{\square \in \llambda}$.

\begin{figure}[htbp]
    \centering
\begin{tikzpicture}[scale=0.8, roundnode/.style={circle,fill,inner sep=2.5pt},refnode/.style={circle,inner sep=0pt},squarednode/.style={rectangle,fill,inner sep=3pt}] 
\node (N0) at (-3,0){\ldots};
\node[roundnode,label=above:{2}] (N1) at (-2,0){};
\node[roundnode,label=above:{2}] (N2) at (-1,0){};
\node[roundnode,label=above:{1}] (N3) at (0,0) {};
\node[roundnode,label=above:{1}] (N4) at (1,0){};
\node[roundnode,label=above:{0}] (N5) at (2,0){};
\node[roundnode,label=above:{0}] (N6) at (3,0){};
\node[squarednode,label=below:{1}] (F1) at (-2,-1){};

\draw[thick, ->] (N0) -- (N1);
\draw[thick, ->] (N1) -- (N2);
\draw[thick, ->] (N2) -- (N3);
\draw[thick, ->] (N3) -- (N4);
\draw[thick, ->] (N4) -- (N5);
\draw[thick, ->] (N5) -- (N6);
\draw[thick, ->] (F1) -- (N1);

\draw[dotted, -](0,0)--(0,5);
\draw[dotted, -](1,0)--(1,5);
\draw[dotted, -](2,0)--(2,5);
\draw[dotted, -](3,0)--(3,5);
\draw[dotted, -](-1,0)--(-1,5);
\draw[dotted, -](-2,0)--(-2,5);
\draw[dotted, -](-3,0)--(-3,5);
\draw[dotted, -](-4,0)--(-4,5);

% The partition (3,2,2,1)
\begin{scope}[shift={(-2,0)}]
\draw[thick,-](0,0)--(-3,3)--(-2,4)--(-1,3)--(1,5)--(2,4)--(3,5)--(4,4)--(0,0);
\draw[thick,-](1,1)--(-2,4);
\draw[thick,-](2,2)--(0,4);
\draw[thick,-](3,3)--(2,4);
\draw[thick,-](-1,1)--(2,4);
\draw[thick,-](-2,2)--(0,4);
\end{scope}

\begin{scope}[shift={(8.1,0)}]

\node (N0) at (-3,0){\ldots};
\node[roundnode,label=above:{2}] (N1) at (-2,0){};
\node[roundnode,label=above:{3}] (N2) at (-1,0){};
\node[roundnode,label=above:{3}] (N3) at (0,0) {};
\node[roundnode,label=above:{3}] (N4) at (1,0){};
\node[roundnode,label=above:{4}] (N5) at (2,0){};
\node[roundnode,label=above:{5}] (N6) at (3,0){};
\node[squarednode,label=below:{6}] (F1) at (3,-1){};

\draw[thick, ->] (N0) -- (N1);
\draw[thick, ->] (N1) -- (N2);
\draw[thick, ->] (N2) -- (N3);
\draw[thick, ->] (N3) -- (N4);
\draw[thick, ->] (N4) -- (N5);
\draw[thick, ->] (N5) -- (N6);
\draw[thick, ->] (F1) -- (N6);

\draw[dotted, -](0,0)--(0,11);
\draw[dotted, -](1,0)--(1,11);
\draw[dotted, -](2,0)--(2,11);
\draw[dotted, -](3,0)--(3,11);
\draw[dotted, -](-1,0)--(-1,11);
\draw[dotted, -](-2,0)--(-2,11);
\draw[dotted, -](-3,0)--(-3,11);
\draw[dotted, -](-4,0)--(-4,11);

% The partition (8)
\begin{scope}[shift={(3,0)}]
\draw[thick,-,fill=gray,opacity=0.5](0,0)--(-5,5)--(-4,6)--(1,1)--(0,0);
\draw[thick,-](0,0)--(-8,8)--(-7,9)--(1,1)--(0,0);
\draw[thick,-](-1,1)--(0,2);
\draw[thick,-](-2,2)--(-1,3);
\draw[thick,-](-3,3)--(-2,4);
\draw[thick,-](-4,4)--(-3,5);
\draw[thick,-](-5,5)--(-4,6);
\draw[thick,-](-6,6)--(-5,7);
\draw[thick,-](-7,7)--(-6,8);
\end{scope}

% The partition (6)
\begin{scope}[shift={(3,2.2)}]
\draw[thick,-,fill=gray,opacity=0.5](0,0)--(-4,4)--(-3,5)--(1,1)--(0,0);
\draw[thick,-](0,0)--(-6,6)--(-5,7)--(1,1)--(0,0);
\draw[thick,-](-1,1)--(0,2);
\draw[thick,-](-2,2)--(-1,3);
\draw[thick,-](-3,3)--(-2,4);
\draw[thick,-](-4,4)--(-3,5);
\draw[thick,-](-5,5)--(-4,6);
\end{scope}

% The partition (5)
\begin{scope}[shift={(3,4.4)}]
\draw[thick,-,fill=gray,opacity=0.5](0,0)--(-3,3)--(-2,4)--(1,1)--(0,0);
\draw[thick,-](0,0)--(-5,5)--(-4,6)--(1,1)--(0,0);
\draw[thick,-](-1,1)--(0,2);
\draw[thick,-](-2,2)--(-1,3);
\draw[thick,-](-3,3)--(-2,4);
\draw[thick,-](-4,4)--(-3,5);
\end{scope}

% The partition (3)
\begin{scope}[shift={(3,6.6)}]
\draw[thick,-,fill=gray,opacity=0.5](0,0)--(-2,2)--(-1,3)--(1,1)--(0,0);
\draw[thick,-](0,0)--(-3,3)--(-2,4)--(1,1)--(0,0);
\draw[thick,-](-1,1)--(0,2);
\draw[thick,-](-2,2)--(-1,3);
\end{scope}

% The partition (1)
\begin{scope}[shift={(3,8.8)}]
\draw[thick,-,fill=gray,opacity=0.5](0,0)--(-1,1)--(0,2)--(1,1)--(0,0);
\end{scope}

% The partition (0)
\begin{scope}[shift={(3,11.2)}]
\node at (0,0){$\emptyset$};
\end{scope}

\end{scope}

\end{tikzpicture}

 \caption{The identification of boxes of $\llambda$ (left) with a subset of the boxes of $\mmu$ (right).}\label{boxdegs}
\end{figure}
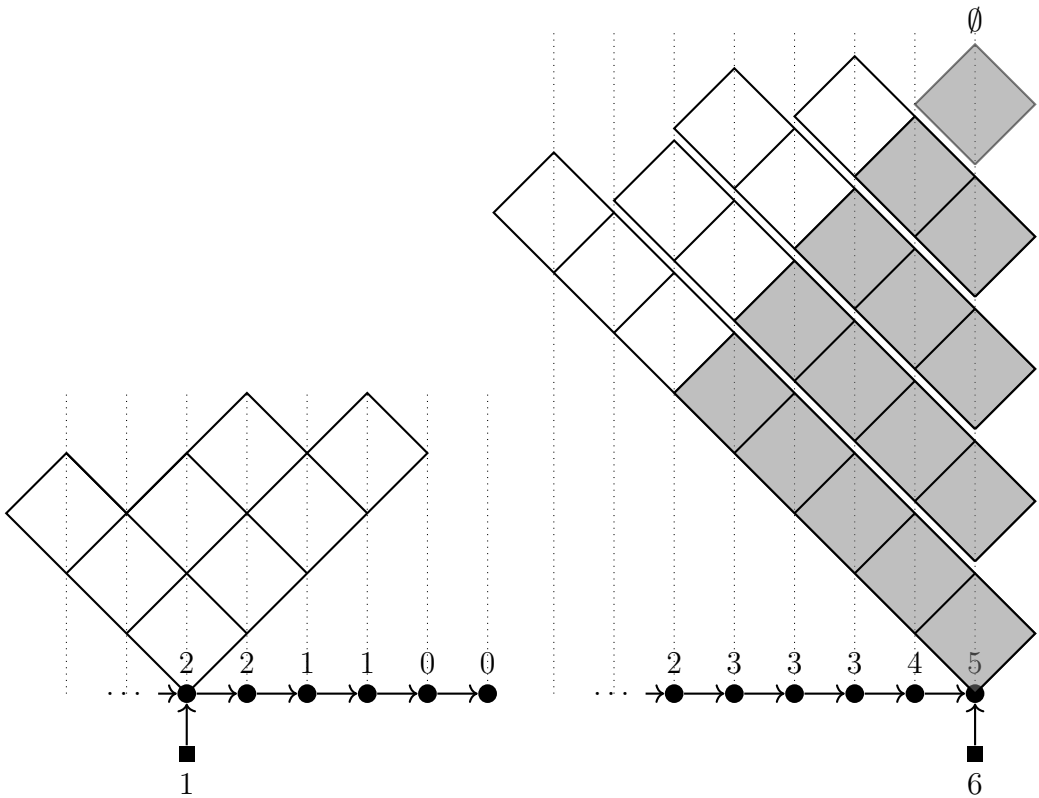

\begin{Lemma}\label{degvanish}
    If $\iota^*\left(C^{\mmu}_{\{d_{\square}\}}(b,\hbar)\right)\neq 0$, then $\{d_{\square}\}_{\square \in \llambda}$ is a collection of admissible degrees for $\llambda$ and $\{d_{\square}\}_{\square \in \mmu} \setminus \{d_{\square}\}_{\square \in \llambda}=\{0\}$.
\end{Lemma}
\begin{proof}
    Suppose $\{d_{\square}\}_{\square \in \mmu}$ is a collection of admissible degrees for $\mmu$ and satisfies $\iota^*(C^{\mmu}_{\{d_{\square}\}}(b,\hbar)\neq 0$. We use the characterization of Proposition \ref{degrees}. We know a-priori that $d_{\square} \geq 0$. The rational function $C^{\mmu}_{\{d_{\square}\}}(b,\hbar)$ is a product of $q$-Pochammer terms. For each $1 \leq i,j \leq n$ and each box $\square$ in $\nu^{j}$ with height $l$, there is a term
    \begin{align*}
        A_{i,j,l}:=\frac{\left(\hbar \frac{c_{j}\hbar^{l}}{c_{i}}\right)_{d_{\square}}}{\left(q \frac{c_{j}\hbar^{l-1}}{c_{i}}\right)_{d_{\square}}}
    \end{align*}
 Then
    $$
\iota^{*}\left(A_{i,j,l}\right) = \frac{\left( \hbar^{j-i+l+1}\right)_{d_{\square}}}{\left(q \hbar^{j-i+l}\right)_{d_{\square}}}
    $$
If $j=i-p$ for $p >1$ and $\square \in \nu^{j}$ has height $p-1$, then the numerator is $(1)_{d'_{\square}}$. For $\iota^{*}\left( A_{i,j,l}\right)$ to be nonzero, we must have $d'_{\square}=0$ for any such box.

This shows that $d_{\square}=0$ whenever $\square$ is a box of height less than or equal to $n+j-1$ in $\nu^{j}$. In other words, $\{d_{\square}\}_{\square \in \mmu} \setminus \{d_{\square}\}_{\square \in \llambda}=\{0\}$.

Now suppose that $\square \in \lambda_{j}$ and $\square' \in \lambda_{j+1}$ satisfy $\delta^{\llambda}_{\square}=\delta^{\llambda}_{\square'}$, or equivalently, $\delta^{\mmu}_{\square}=\delta^{\mmu}_{\square'}+1$. We must show that $\iota^*\left(C_{\{d_{\square}\}}(b,\hbar)\right)\neq 0 \implies d_{\square'}-d_{\square}\geq 0$. A term in $C_{\{d_{\square}\}}(b,\hbar)$ is
$$
B_{\square,\square'}:=\frac{\left(\hbar \frac{c_{j+1}}{c_{j} \hbar }\right)_{d_{\square'}-d_{\square}}}{\left(q \frac{c_{j+1}}{c_{j}\hbar }\right)_{d_{\square'}-d_{\square}}}
$$
And
$$
\iota^*\left(B_{\square,\square'}\right)=\frac{\left( \hbar \right)_{d_{\square'}-d_{\square}}}{\left(q  \right)_{d_{\square'}-d_{\square}}}
$$
which is nonzero if and only if $d_{\square'} - d_{\square} \geq 0$.

A similar argument applies for $\square,\square' \in \lambda_{j}$ such that $\delta^{\llambda}_{\square'}=\delta^{\llambda}_{\square}+1$. Thus $\{d_{\square}\}_{\square \in \llambda}$ is a collection of admissible degrees for $\llambda$.
\end{proof}

\begin{Lemma}\label{boxwts}
    If $\square \in \lambda_{j}$, then $\delta^{\llambda}_{\square}=\delta^{\mmu}_{\square}+j-n$. If $\square \in \llambda \setminus \lambda$, then $\delta^{\llambda}_{\square}=\delta^{\mmu}_{\square}$.
\end{Lemma}
\begin{proof}
    This follows straightforwardly from the definitions.
\end{proof}

\subsection{Comparison of localization terms}

Next we start comparing $\iota^{*}\left(C^{\mmu}_{\{d_{\square}\}}(b,\hbar)\right)$ with $C^{\llambda}_{\{d_{\square}\}}(a,\hbar)$. We define $X^{\mmu},Y^{\mmu}$, and $Z^{\mmu}$ by the following:
% \begin{multline*}
%       C_{\{d_{\square}\}}(b,\hbar)= \underbrace{\prod_{\substack{\square, \square' \in \mmu \\ \gamma(\square')=\gamma(\square)+1}} \frac{\left( \frac{b_{\square'}}{b_{\square} } \hbar^{\delta^{\mmu}_{\square'}-\delta^{\mmu}_{\square}+1}\right)_{d_{\square'}-d_{\square}}}{\left(q  \frac{b_{\square'}}{b_{\square} }  \hbar^{\delta^{\mmu}_{\square'}-\delta^{\mmu}_{\square}} \right)_{d_{\square'}-d_{\square}}} \left( -\frac{q}{\hbar^{1/2}}\right)^{d_{\square'}-d_{\square}}}_{X^{\mmu}} \\
%        \underbrace{\prod_{\square \in \mmu} \prod_{j=1}^{\dw_{\gamma(\square)}} \frac{\left( \frac{b_{\square}}{b_{\gamma(\square),j}} \hbar^{\delta^{\mmu}_{\square}+1} \right)_{d_{\square}}}{\left(q \frac{b_{\square}}{b_{\gamma(\square),j}} \hbar^{\delta^{\mmu}_{\square}}\right)_{d_{\square}}} \left( -\frac{q}{\hbar^{1/2}}\right)^{d_{\square}}}_{Y^{\mmu}} \\ \underbrace{\prod_{\substack{\square, \square' \in \mmu \\ \gamma(\square')=\gamma(\square)}} \frac{\left(q \frac{b_{\square'}}{b_{\square}} \hbar^{\delta^{\mmu}_{\square'}-\delta^{\mmu}_{\square}}\right)_{d_{\square'}-d_{\square}}}{\left( \frac{b_{\square'}}{b_{\square}} \hbar^{\delta^{\mmu}_{\square'}-\delta^{\mmu}_{\square}+1}\right)_{d_{\square'}-d_{\square}}} \left(-\frac{q}{\hbar^{1/2}} \right)^{d_{\square}-d_{\square'}}}_{Z^{\mmu}}
%     \end{multline*}

\begin{multline*}
      C^{\mmu}_{\{d_{\square}\}}(b,\hbar)= \underbrace{\prod_{\substack{\square, \square' \in \mmu \\ \gamma(\square')=\gamma(\square)+1}} \left\{ \frac{b_{\square'}}{b_{\square}} \hbar^{\delta^{\mmu}_{\square'}-\delta^{\mmu}_{\square}} \right\}_{d_{\square'}-d_{\square}} }_{X^{\mmu}} \\
       \underbrace{\prod_{\square \in \mmu} \prod_{j=1}^{\dw_{\gamma(\square)}} \left\{\frac{b_{\square}}{b_{\gamma(\square),j}} \hbar^{\delta^{\mmu}_{\square}} \right\}_{d_{\square}}}_{Y^{\mmu}} 
       \underbrace{\prod_{\substack{\square, \square' \in \mmu \\ \gamma(\square')=\gamma(\square)}} \left\{\frac{b_{\square'}}{b_{\square}} \hbar^{\delta^{\mmu}_{\square'}-\delta^{\mmu}_{\square}} \right\}_{d_{\square'}-d_{\square}}^{-1}}_{Z^{\mmu}}
    \end{multline*}

\begin{Lemma}
    \begin{align}
       &\frac{ \iota^{*}\left(X^{\mmu}\right)}{X^{\llambda}}=\prod_{j=1}^{n} \prod_{\substack{\square \in (\nu \setminus \lambda)_{j} \\ \square' \notin \nu \\ \gamma(\square')=\gamma(\square)+1}} \left\{\frac{a_{\square'}}{a} \hbar^{\delta^{\mmu}_{\square'}-\delta^{\mmu}_{\square}-j+n} \right\}_{d_{\square'}} 
       \prod_{l=1}^{n} \prod_{\substack{\square \notin \nu \\ \square' \in (\nu\setminus \lambda)_{l} \\ \gamma(\square')=\gamma(\square)+1}} \left\{\frac{a}{a_{\square}} \hbar^{\delta^{\mmu}_{\square'}-\delta^{\mmu}_{\square}+l-n}\right\}_{-d_{\square}} \nonumber \\
       & \prod_{j,l=1}^{n} \prod_{\substack{\square \in (\nu\setminus\lambda)_{j} \\ \square' \in \lambda_{l} \\ \gamma(\square')=\gamma(\square)+1}} \left\{\hbar^{\delta^{\mmu}_{\square'}-\delta^{\mmu}_{\square} +l-j}\right\}_{d_{\square'}}
       \prod_{j,l=1}^{n} \prod_{\substack{\square \in \lambda_{j} \\ \square' \in (\nu \setminus\lambda)_{l} \\ \gamma(\square')=\gamma(\square)+1}} \left\{\hbar^{\delta^{\mmu}_{\square'}-\delta^{\mmu}_{\square}+l-j} \right\}_{-d_{\square}} \label{Xs}
    \end{align}
\end{Lemma}
\begin{proof}
       We split the term as
    $$
X^{\mmu}=\prod_{\square, \square' \in \mmu} X^{\mmu}_{\square,\square'}, \quad X^{\mmu}_{\square,\square'}=\left\{ \frac{b_{\square'}}{b_{\square}} \hbar^{\delta^{\mmu}_{\square'}-\delta^{\mmu}_{\square}} \right\}_{d_{\square'}-d_{\square}} 
$$
and consider $\iota^*(X^{\mmu})$. Assuming that $\gamma(\square')=\gamma(\square)+1$, we break the possibilities into four cases:
\begin{enumerate}
    \item Suppose $\square,\square' \notin \nu$. By Lemma \ref{boxwts}, we have 
    $$
    \iota^{*}(X^{\mmu}_{\square,\square'})=\left\{\frac{a_{\square'}}{a_{\square}} \hbar^{\delta^{\llambda}_{\square'}-\delta^{\llambda}_{\square}} \right\}_{d_{\square'}-d_{\square}}=X^{\llambda}_{\square,\square'}
    $$
    \item Suppose $\square' \notin \nu$ and $\square \in \nu^{j}$. If $\square \in \lambda_{j}\subset \nu^{j}$, then 
    \begin{align*}
\iota^{*}(X^{\mmu}_{\square,\square'})&= \left\{\frac{a_{\square'}}{a} \hbar^{n-j+\delta^{\mmu}_{\square'}-\delta^{\mmu}_{\square}} \right\}_{d_{\square'}-d_{\square}} \\
&=\left\{\frac{a_{\square'}}{a} \hbar^{\delta^{\llambda}_{\square'}-\delta^{\llambda}_{\square}} \right\}_{d_{\square'}-d_{\square}} \\
&= X^{\llambda}_{\square,\square'}
        \end{align*}
        where the second equality follows from Lemma \ref{boxwts}. If $\square \in   (\nu \setminus \lambda)_{j}$, then we get
        \begin{align*}
            \iota^{*}(X^{\mmu}_{\square,\square'})= \left\{\frac{a_{\square'}}{a} \hbar^{\delta^{\mmu}_{\square'}-\delta^{\mmu}_{\square}-j+n} \right\}_{d_{\square'}}
        \end{align*}
        since $d_{\square}=0$.
        
        \item Similarly, if $\square \notin \nu$ and $\square' \in \nu^{l}$, we obtain either $X^{\llambda}_{\square,\square'}$ or the extra terms
        \begin{align*}
           \iota(X^{\mmu}_{\square,\square'})=\left\{\frac{a}{a_{\square}} \hbar^{l-n+\delta^{\mmu}_{\square'}-\delta^{\mmu}_{\square}} \right\}_{-d_{\square}} 
        \end{align*}
        which arise from the case when $\square' \in (\nu\setminus\lambda)_{l}$.
        
        \item Suppose $\square \in \nu^{j}$ and $\square' \in \nu^{l}$. If $\square \in \lambda_{j} \subset \nu^{j}$ and $\square' \in \lambda_{l} \subset \nu^{l}$, then
        \begin{align*}
            \iota^{*}(X^{\mmu}_{\square,\square'})&=\left\{\hbar^{l-j+\delta^{\mmu}_{\square'}-\delta^{\mmu}_{\square}}\right\}_{d_{\square'}-d_{\square}}  \\
            &=\left\{\hbar^{\delta^{\llambda}_{\square'}-\delta^{\llambda}_{\square}} \right\}_{d_{\square'}-d_{\square}} \\
            &= X^{\llambda}_{\square,\square'}
        \end{align*}
        If $\square \in (\nu\setminus \lambda)_{j}$ and $\square' \in \lambda_{l} \subset \nu_{l}$, then 
        \begin{align*}
            \iota^{*}(X^{\mmu}_{\square,\square'})=\left\{\hbar^{l-j+\delta^{\mmu}_{\square'}-\delta^{\mmu}_{\square}} \right\}_{d_{\square'}}
        \end{align*}
        If $\square' \in (\nu\setminus\lambda)_{l}$ and $\square \in \lambda_{j}$, then 
        \begin{align*}
           \iota^{*}(X^{\mmu}_{\square,\square'})=\left\{\hbar^{l-j+\delta^{\mmu}_{\square'}-\delta^{\mmu}_{\square}}\right\}_{-d_{\square}}
        \end{align*}
        If $\square' \in (\nu\setminus\lambda)_{l}$ and $\square \in (\nu\setminus\lambda)_{j}$, then $d_{\square'}=d_{\square}=0$. So $\iota^{*}(X^{\mmu}_{\square,\square'})=1$.
    
\end{enumerate}
\end{proof}

\begin{Lemma}
     \begin{align}
             &\frac{ \iota^{*}\left(Y^{\mmu}\right)}{Y^{\llambda}}=\prod_{j=1}^{n} \prod_{\substack{\square \notin \nu \\ \gamma(\square)=m}} \left\{\frac{a_{\square}}{a} \hbar^{n-j}\right\}_{d_{\square}}  \prod_{j,l=1}^{n}\prod_{\substack{\square \in \lambda_{l}  \\ \gamma(\square)=m}} \left\{\hbar^{l-j} \right\}_{d_{\square}} \label{Ys}
    \end{align}
\end{Lemma}
\begin{proof}
    
We split $Y^{\mmu}$ as
$$
Y^{\mmu}=\prod_{\square \in \mmu} \prod_{j=1}^{\dw_{\gamma(\square)}} Y^{\mmu}_{\square,j}, \quad Y^{\mmu}_{\square,j}= \left\{\frac{b_{\square}}{b_{\gamma(\square),j}} \hbar^{\delta^{\mmu}_{\square}} \right\}_{d_{\square}}
$$

% We can assume that $\square \notin \nu\setminus\lambda$, since we showed in the proof of Lemma \ref{degvanish} that boxes in $\nu\setminus\lambda$ must have $d_{\square}=0$ to contribute to $\iota^{*}\left(C_{\{d'_{\square}\}}(b,\hbar)\right)$; such boxes contribute 1 to $\iota^{*}(Y^{\mmu})$. 
Assuming that $\square \in \mmu$ and $1 \leq j \leq \dw_{\gamma(\square)}$, we have five possibilities:
\begin{enumerate}
    \item Suppose $\square \notin \nu$. Suppose also that either $\gamma(\square)\neq m$ or $\gamma(\square)=m$ and $1\leq j \leq \dw_{m}$. Then
    $$
\iota^{*}(Y^{\mmu}_{\square,j})=\left\{\frac{a_{\square}}{a_{\gamma(\square),j}} \hbar^{\delta^{\llambda}_{\square}} \right\}_{d_{\square}} = Y^{\llambda}_{\square,j}
    $$
      \item Suppose $\square \notin \nu$. Suppose also that $\gamma(\square)=m$ and $\dw_{m}+1\leq j \leq \dw_{m}+n$, and let $i=j-\dw_{m}$. We necessarily have $\delta^{\mmu}_{\square}=\delta^{\llambda}_{\square}=0$. Then
    \begin{align}
\iota^{*}(Y^{\mmu}_{\square,j})=\left\{\frac{a_{\square}}{a}\hbar^{n-i} \right\}_{d_{\square}}
    \end{align}
    \item Suppose $\square \in \lambda_{l} \subset \nu^{l}$. Suppose also that $\gamma(\square)\neq m$ or $\gamma(\square)=m$ and $1 \leq j \leq \dw_{m}$. Then 
    \begin{align*}
\iota^{*}(Y^{\mmu}_{\square,j})&=\left\{\frac{a}{a_{\gamma(\square),j}} \hbar^{l-n+\delta^{\mmu}_{\square}} \right\}_{d_{\square}} \\
&=\left\{\frac{a}{a_{\gamma(\square),j}} \hbar^{\delta^{\llambda}_{\square}}  \right\}_{d_{\square}} \\
&= Y^{\llambda}_{\square,j}
    \end{align*}
    where the second equality follows by Lemma \ref{boxwts}. 
    \item Suppose $\square \in \lambda_{l}\subset \nu^{l}$ and $\gamma(\square)=m$. Such a box must satisfy $\delta^{\mmu}_{\square}=\delta^{\llambda}_{\square}=0$. Suppose also that $j$ satisfies $\dw_{m}+1 \leq j \leq \dw_{m}+n$. Let $i=j-\dw_{m}$. Then 
    \begin{align}
        \iota^{*}(Y^{\mmu}_{\square,j})&=\left\{\hbar^{l-i} \right\}_{d_{\square}} 
    \end{align}
    \item The terms from $\square \in (\nu\setminus\lambda)_{l}$ are all $1$, since $d_{\square}=0$ for such a box.
\end{enumerate}

\end{proof}

\begin{Lemma}
    \begin{align}
       &\frac{ \iota^{*}\left(Z^{\mmu}\right)}{Z^{\llambda}}=\prod_{j=1}^{n} \prod_{\substack{\square \in (\nu \setminus \lambda)_{j} \\ \square' \notin \nu \\ \gamma(\square')=\gamma(\square)}} \left\{\frac{a_{\square'}}{a} \hbar^{\delta^{\mmu}_{\square'}-\delta^{\mmu}_{\square}-j+n} \right\}_{d_{\square'}}^{-1} 
       \prod_{j=1}^{n} \prod_{\substack{\square' \in (\nu \setminus \lambda)_{l} \\ \square \notin \nu \\ \gamma(\square')=\gamma(\square)}} \left\{\frac{a}{a_{\square}} \hbar^{\delta^{\mmu}_{\square'}-\delta^{\mmu}_{\square}+l-n} \right\}_{-d_{\square}}^{-1} \nonumber \\
     &  \prod_{j,l=1}^{n} \prod_{\substack{\square \in (\nu\setminus \lambda)_{j} \\ \square' \in \lambda_{l} \\ \gamma(\square')=\gamma(\square)}} \left\{ \hbar^{l-j+\delta^{\mmu}_{\square'}-\delta^{\mmu}_{\square}}\right\}_{d_{\square'}}^{-1} 
     \prod_{j,l=1}^{n} \prod_{\substack{\square \in \lambda_{j} \\ \square' \in (\nu \setminus \lambda)_{l} \\ \gamma(\square')=\gamma(\square)}} \left\{\hbar^{l-j +\delta^{\mmu}_{\square'}-\delta^{\mmu}_{\square}} \right\}_{-d_{\square}}^{-1} \label{Zs}
    \end{align}
\end{Lemma}
\begin{proof}
We write 
$$
Z^{\mmu}=\prod_{\substack{\square,\square' \in \mmu \\ \gamma(\square')=\gamma(\square)}} Z^{\mmu}_{\square,\square'}, \quad Z^{\mmu}_{\square,\square'}= \left\{\frac{b_{\square'}}{b_{\square}} \hbar^{\delta^{\mmu}_{\square'}-\delta^{\mmu}_{\square}} \right\}_{d_{\square'}-d_{\square}}^{-1}
$$
Assuming that $\gamma(\square)=\gamma(\square')$, we break the possibilities into four cases:
\begin{enumerate}
    \item Suppose that $\square,\square' \notin \nu$. Then $\iota^{*}(Z^{\mmu}_{\square,\square'})=Z^{\llambda}_{\square,\square'}$.
    \item Suppose $\square \in \nu^{j}$ and $\square' \notin \nu$. Then
    \begin{align*}
       \iota^{*}(Z^{\mmu}_{\square,\square'})= \left\{ \frac{a_{\square'}}{a} \hbar^{\delta^{\mmu}_{\square'}-\delta^{\mmu}_{\square}+n-j}\right\}_{d_{\square'}-d_{\square}}^{-1}
    \end{align*}
    If $\square \in \lambda_{j} \subset \nu_{j}$, then
    \begin{align*}
        \iota^{*}(Z^{\mmu}_{\square,\square'})= \left\{\frac{a_{\square'}}{a} \hbar^{\delta^{\llambda}_{\square'}-\delta^{\llambda}_{\square}}\right\}_{d_{\square'}-d_{\square}}^{-1} = Z^{\llambda}_{\square,\square'}
    \end{align*}
    If $\square \in (\nu\setminus\lambda)_{j}$, then
    \begin{align*}
       \iota^{*}(Z^{\mmu}_{\square,\square'})=\left\{\frac{a_{\square'}}{a}\hbar^{\delta^{\mmu}_{\square'}-\delta^{\mmu}_{\square}+n-j}\right\}_{d_{\square'}}^{-1}
    \end{align*}
    since $d_{\square}=0$.
    \item Similarly, if $\square' \in (\nu\setminus\lambda)_{l}$ and $\square \notin \nu$ then
    \begin{align*}
         \iota^{*}(Z^{\mmu}_{\square,\square'})= \left\{\frac{a}{a_{\square}} \hbar^{\delta^{\mmu}_{\square'}-\delta^{\mmu}_{\square}+l-n} \right\}_{-d_{\square}}^{-1}
    \end{align*}
    \item Suppose $\square \in \nu^{j}$ and $\square' \in \nu^{l}$. Then 
    \begin{align*}
        \iota^{*}(Z^{\mmu}_{\square,\square'})=\left\{\hbar^{l-j+\delta^{\mmu}_{\square'}-\delta^{\mmu}_{\square}} \right\}_{d_{\square'}-d_{\square}}^{-1} 
    \end{align*}
    If $\square \in \lambda_{j}$ and $\square' \in \lambda_{l}$, then $\iota^{*}(Z^{\mmu}_{\square,\square'})=Z^{\llambda}_{\square,\square'}$, and similarly if we switch the roles of $\square$ and $\square'$.
    
    If $\square \in (\nu\setminus \lambda)_{j}$ and $\square'\in \lambda_{l}$, then
    \begin{align*}
       \iota^{*}(Z^{\mmu}_{\square,\square'})=\left\{\hbar^{l-j+\delta^{\mmu}_{\square'}-\delta^{\mmu}_{\square}} \right\}_{d_{\square'}}^{-1} 
    \end{align*}

    With the roles of $\square$ and $\square'$ reversed, we get
    \begin{align*}
        \iota^{*}(Z^{\mmu}_{\square,\square'})=\left\{\hbar^{l-j+\delta^{\mmu}_{\square'}-\delta^{\mmu}_{\square}}\right\}_{-d_{\square}}^{-1}
    \end{align*}

    If  $\square \in (\nu \setminus \lambda)_{j}$ and $\square' \in (\nu\setminus \lambda)_{l}$, then we get $1$.
    
\end{enumerate}    
\end{proof}

\begin{Theorem}
$$
\iota^*\left(C^{\mmu}_{\{d_{\square}\}}(b,\hbar)\right)= C^{\llambda}_{\{d_{\square}\}}(a,\hbar)\prod_{\substack{\square \in \mmu \\ \gamma(\square)=m}} q^{(n-1)d_{\square}}\prod_{l=1}^{n}\prod_{\substack{\square \in \mmu \\ \gamma(\square)=k+l}} q^{-d_{\square}}
$$
\end{Theorem}

\begin{proof}
What we have shown so far is that $\iota^{*}\left(C^{\mmu}_{\{d'_{\square}\}}(b,\hbar)\right)/C^{\llambda}_{\{d_{\square}\}}(a,\hbar)$ is a product of terms appearing in the previous three lemmas.  In what follows, we will often use the identity $\{ x \}_{-d}=\{\hbar^{-1} x^{-1} \}_{d} q^{-d}$ for $d\geq 0$, which follows by direct computation.

First, we locate the terms in $C^{\llambda}_{\{d_{\square}\}}(a,\hbar)$ corresponding to the framing at vertex $k+1$. Let $\square' \in (\nu\setminus\lambda)_{1}$ be the box such that $\gamma(\square')=k+2$ and $\delta^{\mmu}_{\square'}=n-2$. Then inside the product in (\ref{Xs}), there are terms
\begin{align*}
& \prod_{\substack{\square \notin \nu\\ \gamma(\square')=k+1 }} \left\{\frac{a}{a_{\square}} \hbar^{-1-\delta^{\mmu}_{\square}} \right\}_{-d_{\square}}\prod_{j=1}^{n}\prod_{\substack{\square \in \lambda_{j}\\ \gamma(\square')=k+1 }} \left\{ \hbar^{n-2-(n-1)+1-j} \right\}_{-d_{\square}} \\
&= \prod_{\substack{\square \notin \nu\\ \gamma(\square')=k+1 }} \left\{\frac{a}{a_{\square}} \hbar^{-1-\delta^{\mmu}_{\square}} \right\}_{-d_{\square}}\prod_{j=1}^{n}\prod_{\substack{\square \in \lambda_{j}\\ \gamma(\square')=k+1 }} \left\{ \hbar^{-j} \right\}_{-d_{\square}} \\
&=\prod_{\substack{\square \notin \nu\\ \gamma(\square')=k+1 }} \left\{\frac{a_{\square}}{a} \hbar^{\delta^{\mmu}_{\square}} \right\}_{d_{\square}} q^{-d_{\square}} \prod_{j=1}^{n}\prod_{\substack{\square \in \lambda_{j}\\ \gamma(\square')=k+1 }} \left\{ \hbar^{j-1} \right\}_{d_{\square}} q^{-d_{\square}} \\
&= \prod_{\substack{\square \in \mmu \\ \gamma(\square)=k+1}} \left\{\frac{a_{\square}}{a} \hbar^{d^{\llambda}_{\square}} \right\}_{d_{\square}} q^{-d_{\square}}
\end{align*}
which, up to the powers of $q$, are exactly the missing framing terms in $V_{\llambda}$.

Next, let $\square\notin \nu$ and suppose that $\gamma(\square)=k+2$. All the contributions of this box to $\iota^{*}\left(C^{\mmu}_{\{d'_{\square}\}}(b,\hbar)\right)/C^{\llambda}_{\{d_{\square}\}}(a,\hbar)$ come from interactions of $\square$ with itself, and then from the two boxes of $\nu\setminus \lambda$ that are one place to the right of $\square$. These terms are
\begin{align*}
   & \left\{\frac{a_{\square}}{a} \hbar^{\delta^{\mmu}_{\square}-(n-2)-1+n} \right\}_{d_{\square}}^{-1}  \left\{\frac{a}{a_{\square}} \hbar^{(n-3)-\delta^{\mmu}_{\square}+1-n} \right\}_{-d_{\square}} \\
    & \left\{\frac{a}{a_{\square}} \hbar^{(n-2)-\delta^{\mmu}_{\square}+1-n} \right\}_{-d_{\square}}^{-1}  \left\{\frac{a}{a_{\square}} \hbar^{(n-3)-\delta^{\mmu}_{\square}+2-n} \right\}_{-d_{\square}} \\
     = & \left\{\frac{a_{\square}}{a} \hbar^{\delta^{\mmu}_{\square}+1} \right\}_{d_{\square}}^{-1}  \left\{\frac{a}{a_{\square}} \hbar^{-2-\delta^{\mmu}_{\square}} \right\}_{-d_{\square}}  \left\{\frac{a}{a_{\square}} \hbar^{-1-\delta^{\mmu}_{\square}} \right\}_{-d_{\square}}^{-1}  \left\{\frac{a}{a_{\square}} \hbar^{-1-\delta^{\mmu}_{\square}} \right\}_{-d_{\square}}  \\
     =& \left\{\frac{a_{\square}}{a} \hbar^{\delta^{\mmu}_{\square}+1} \right\}_{d_{\square}}^{-1}  \left\{\frac{a}{a_{\square}} \hbar^{-2-\delta^{\mmu}_{\square}} \right\}_{-d_{\square}}  \\
     &=\left\{\frac{a_{\square}}{a} \hbar^{\delta^{\mmu}_{\square}+1} \right\}_{d_{\square}}^{-1}  \left\{\frac{a_{\square}}{a} \hbar^{1+\delta^{\mmu}_{\square}} \right\}_{d_{\square}} q^{-d_{\square}} \\
     &= q^{-d_{\square}}
\end{align*}

Next suppose that $\square \notin \nu$ satisfies $\gamma(\square)=k+j$ for $2\leq j \leq m$. Then the terms involving $\square$ in (\ref{Xs}), (\ref{Ys}), and (\ref{Zs}) are
\begin{align*}
  & \prod_{l=1}^{n}\Bigg( \prod_{\substack{\square' \in (\nu\setminus\lambda)_{l} \\ \gamma(\square')=k+j}} \left\{\frac{a_{\square}}{a}\hbar^{\delta^{\mmu}_{\square}-\delta^{\mmu}_{\square'}-l+n} \right\}_{d_{\square}}^{-1} \prod_{\substack{\square' \in (\nu\setminus\lambda)_{l} \\ \gamma(\square')=k+j}}  \left\{\frac{a}{a_{\square}} \hbar^{\delta^{\mmu}_{\square'}-\delta^{\mmu}_{\square}+l-n} \right\}_{-d_{\square}}^{-1}\\
     &\prod_{\substack{\square' \in (\nu\setminus\lambda)_{l} \\ \gamma(\square')=k+j-1}}  \left\{\frac{a_{\square}}{a} \hbar^{\delta^{\mmu}_{\square}-\delta^{\mmu}_{\square'}-l+n} \right\}_{d_{\square}}  \prod_{\substack{\square' \in (\nu\setminus\lambda)_{l} \\ \gamma(\square')=k+j+1}} \left\{\frac{a}{a_{\square}} \hbar^{\delta^{\mmu}_{\square'}-\delta^{\mmu}_{\square}+l-n} \right\}_{-d_{\square}} \Bigg) \\
     =& \prod_{l=1}^{j-1}\left\{\frac{a_{\square}}{a} \hbar^{\delta^{\mmu}_{\square}+j-l} \right\}_{d_{\square}}^{-1} \prod_{l=2}^{j} \left\{\frac{a}{a_{\square}} \hbar^{-1-j+l-\delta^{\mmu}_{\square}} \right\}_{-d_{\square}}^{-1}  \\
    & \prod_{l=2}^{j-1} \left\{ \frac{a_{\square}}{a}\hbar^{\delta^{\mmu}_{\square}+j-l}\right\}_{d_{\square}} \prod_{l=1}^{j} \left\{\frac{a}{a_{\square}} \hbar^{-1+l-j-\delta^{\mmu}_{\square}} \right\}_{-d_{\square}} \\
    =& \left\{\frac{a}{a_{\square}} \hbar^{-j-\delta^{\mmu}_{\square}} \right\}_{-d_{\square}}  \left\{\frac{a_{\square}}{a} \hbar^{\delta^{\mmu}_{\square}+j-1} \right\}_{d_{\square}}^{-1} \\
     = &q^{-d_{\square}}
\end{align*}

Next, suppose that $\square \notin \nu$ satisfies $\gamma(\square)=m$. The terms in (\ref{Xs}), (\ref{Ys}), and (\ref{Zs}) involving $\square$ are 
\begin{align*}
   &\prod_{l=1}^{n-1}\left\{\frac{a_{\square}}{a} \hbar^{n-l}\right\}_{d_{\square}}^{-1} \prod_{l=2}^{n} \left\{ \frac{a}{a_{\square}} \hbar^{-1-n+l}\right\}_{-d_{\square}}^{-1} \\
   & \prod_{l=2}^{n-1} \left\{\frac{a_{\square}}{a} \hbar^{n-l} \right\}_{d_{\square}} \prod_{l=1}^{n} \left\{\frac{a_{\square}}{a} \hbar^{n-l} \right\}_{d_{\square}} \\
    =&\prod_{l=2}^{n} q^{d_{\square}} \\
    =& q^{(n-1)d_{\square}}
\end{align*}

This accounts for all the terms in (\ref{Xs}), (\ref{Ys}), and (\ref{Zs}) involving $\square \notin \nu$. The computation for the remaining terms (i.e. the terms involving only $\hbar$ and $q$) is similar, and we omit it.

\end{proof}

 %Z^{\mmu}_{\square,\square'}= \frac{\left(q \frac{b_{\square'}}{b_{\square}} \hbar^{\delta^{\mmu}_{\square'}-\delta^{\mmu}_{\square}}\right)_{d'_{\square'}-d'_{\square}}}{\left( \frac{b_{\square'}}{b_{\square}} \hbar^{\delta^{\mmu}_{\square'}-\delta^{\mmu}_{\square}+1}\right)_{d'_{\square'}-d'_{\square}}} \left(-\frac{q}{\hbar^{1/2}} \right)^{d'_{\square'}-d'_{\square}}

% \frac{\left( \frac{b_{\square}}{b_{\gamma(\square),j}} \hbar^{\delta^{\mmu}_{\square}+1} \right)_{d'_{\square}}}{\left(q \frac{b_{\square}}{b_{\gamma(\square),j}} \hbar^{\delta^{\mmu}_{\square}}\right)_{d'_{\square}}} \left( -\frac{q}{\hbar^{1/2}}\right)^{d'_{\square}}

\printbibliography

@article{MO,
author = {Maulik, Davesh and Okounkov, Andrei},
year = {2012},
month = {11},
pages = {},
title = {Quantum Groups and Quantum Cohomology},
volume = {408},
journal = {Astérisque}
}

@incollection{pcmilect,
      author         = "Okounkov, Andrei",
      title          = "{Lectures on K-theoretic computations in enumerative
                        geometry}",
    booktitle= {Geometry of Moduli Spaces and Representation Theory},
    publisher={American Mathematical Society},
    series={IAS/Park City Mathematics Series},
      year           = "2017",
      volume={24}
}

@article {EV,
    AUTHOR = {Etingof, P. and Varchenko, A.},
     TITLE = {Dynamical {W}eyl groups and applications},
   JOURNAL = {Adv. Math.},
  FJOURNAL = {Advances in Mathematics},
    VOLUME = {167},
      YEAR = {2002},
    NUMBER = {1},
     PAGES = {74--127},
      ISSN = {0001-8708},
     CODEN = {ADMTA4},
   MRCLASS = {17B10 (17B20 17B37 39A12 81R50)},
  MRNUMBER = {1901247 (2003d:17004)},
MRREVIEWER = {Anjan Kundu},
       DOI = {10.1006/aima.2001.2034}
}

@preamble{
   "\def\cprime{$'$} "
}

@article {NakALE,
    AUTHOR = {Nakajima, Hiraku},
     TITLE = {Instantons on {ALE} spaces, quiver varieties, and
              {K}ac-{M}oody algebras},
   JOURNAL = {Duke Math. J.},
  FJOURNAL = {Duke Mathematical Journal},
    VOLUME = {76},
      YEAR = {1994},
    NUMBER = {2},
     PAGES = {365--416},
    %   ISSN = {0012-7094},
     CODEN = {DUMJAO},
   MRCLASS = {53C25 (17B67 58D27 58E15)},
  MRNUMBER = {1302318 (95i:53051)},
MRREVIEWER = {Andrew Dancer},
       DOI = {10.1215/S0012-7094-94-07613-8},
      
}

@article {NakQv,
    AUTHOR = {Nakajima, Hiraku},
     TITLE = {Quiver varieties and {K}ac-{M}oody algebras},
   JOURNAL = {Duke Math. J.},
  FJOURNAL = {Duke Mathematical Journal},
    VOLUME = {91},
      YEAR = {1998},
    NUMBER = {3},
     PAGES = {515--560},
      ISSN = {0012-7094},
     CODEN = {DUMJAO},
   MRCLASS = {17B67 (14D25 16G20 17B35 53C25 58F05)},
  MRNUMBER = {1604167 (99b:17033)},
MRREVIEWER = {Michael M. Kapranov},
       DOI = {10.1215/S0012-7094-98-09120-7},
    
}

@article{NO,
author = {Nekrasov, N. and Okounkov, Andrei},
year = {2016},
pages = {320-369},
title = {Membranes and Sheaves},
volume = {3},
journal = {Algebraic Geometry},
doi = {10.14231/AG-2016-015}
}

@ARTICLE{OkBethe,
   author = {{Aganagic}, Mina and {Okounkov}, Andrei},
    title = "{Quasimap counts and Bethe eigenfunctions}",
  journal = {Mosc. Math. J.},
     year = 2017,
     volume = 17,
     pages = {565-600},
   adsurl = {http://adsabs.harvard.edu/abs/2017arXiv170408746A}
}

@article{AOElliptic,
       author = {{Aganagic}, Mina and {Okounkov}, Andrei},
        title = "{Elliptic stable envelopes}",
      journal = {J. Amer. Math. Soc.},
     keywords = {Mathematics - Algebraic Geometry, High Energy Physics - Theory, Mathematical Physics, Mathematics - Representation Theory},
     volume={34},
         year = "2021",
        month = "4",
        pages={79-133},
% archivePrefix = {arXiv},
%       eprint = {1604.00423v4},
%  primaryClass = {math.AG},
       adsurl = {https://ui.adsabs.harvard.edu/abs/2016arXiv160400423A},
      adsnote = {Provided by the SAO/NASA Astrophysics Data System}
}

@incollection{GinzburgLectures,
    AUTHOR = {Ginzburg, Victor},
     TITLE = {Lectures on {N}akajima's quiver varieties},
 BOOKTITLE = {Geometric methods in representation theory. {I}},
    SERIES = {S\'emin. Congr.},
    VOLUME = {24},
     PAGES = {145--219},
 PUBLISHER = {Soc. Math. France, Paris},
      YEAR = {2012},
   MRCLASS = {14L24 (16G20 17B67)},
  MRNUMBER = {3202703},
MRREVIEWER = {Xueqing Chen},
}

@article{Pushk1,
author = {Pushkar, Petr and Smirnov, Andrey and Zeitlin, Anton},
year = {2016},
month = {12},
title = {Baxter Q-operator from quantum K-theory},
volume = {360},
journal = {Adv. Math.},
doi = {10.1016/j.aim.2019.106919}
}

@article {kirv,
	AUTHOR = {McGerty, Kevin and Nevins, Thomas},
	TITLE = {Kirwan surjectivity for quiver varieties},
	JOURNAL = {Invent. Math.},
	VOLUME = {212},
	YEAR = {2018},
	month={04},
	NUMBER = {1},
	PAGES = {161--187},
	doi = {10.1007/s00222-017-0765-x}
}

@article{MirSym1,
    author = {Rimányi, Richárd and Smirnov, Andrey and Zhou, Zijun and Varchenko, Alexander},
    title = "{Three-Dimensional Mirror Symmetry and Elliptic Stable Envelopes}",
    journal = {International Mathematics Research Notices},
    volume = {2022},
    number = {13},
    pages = {10016-10094},
    year = {2021},
    month = {02},
    issn = {1073-7928},
    doi = {10.1093/imrn/rnaa389}
}

@article{MirSym2,
       author = {{Rim{\'a}nyi}, Rich{\'a}rd and {Smirnov}, Andrey and {Varchenko}, Alexander and {Zhou}, Zijun},
        title = "{Three-Dimensional Mirror Self-Symmetry of the Cotangent Bundle of the Full Flag Variety}",
      journal = {SIGMA},
         year = 2019,
        month = nov,
       volume = {15},
          eid = {093},
        pages = {093},
          doi = {10.3842/SIGMA.2019.093},
% archivePrefix = {arXiv},
%       eprint = {1906.00134},
%  primaryClass = {math.AG},
}

@book{hypergeo,
	AUTHOR = {{Gasper}, George and {Rahman}, Mizan},
	TITLE = {Basic Hypergeometric Series},
	PUBLISHER = {Cambridge University Press },
	YEAR = {1990},
	ISBN = {978-0521-35049-5}
}

@article{dinksmir,
author = {Dinkins, Hunter and Smirnov, Andrey},
year = {2020},
month = {09},
pages = {},
title = {Characters of tangent spaces at torus fixed points and 3d-mirror symmetry},
volume = {110},
journal = {Letters in Mathematical Physics},
doi = {10.1007/s11005-020-01292-y}
}

@article{qm,
title = {Stable quasimaps to GIT quotients},
journal = {J. Geom. Phys.},
volume = {75},
pages = {17 - 47},
year = {2014},
author = {Ionuţ Ciocan-Fontanine and Bumsig Kim and Davesh Maulik},
}

@article{KorZeit,
    author = "Koroteev, Peter and Zeitlin, Anton M.",
    title = "{qKZ/tRS Duality via Quantum K-Theoretic Counts}",
    % eprint = "1802.04463",
    % archivePrefix = "arXiv",
    % primaryClass = "math.AG",
    doi = "10.4310/MRL.2021.v28.n2.a5",
    journal = "Math. Res. Lett.",
    volume = "28",
    number = "2",
    pages = "435--470",
    year = "2021"
}

@article{tRSKor,
  title={A-type Quiver Varieties and ADHM Moduli Spaces},
  author={P. Koroteev},
  journal={Communications in Mathematical Physics},
  year={2018},
  volume={381},
  pages={175-207}
}

@article{dinksmir2,
    author = {Dinkins, Hunter and Smirnov, Andrey},
    title = "{Quasimaps to Zero-Dimensional $A_{\infty}$-Quiver Varieties}",
    journal = {International Mathematics Research Notices},
    year = {2020},
    month = {06},
    doi = {10.1093/imrn/rnaa129}
}

@article{dinksmir3,
title = {Capped vertex with descendants for zero dimensional $A_{\infty}$ quiver varieties},
journal = {Advances in Mathematics},
volume = {401},
pages = {108324},
year = {2022},
doi = {10.1016/j.aim.2022.108324},
author = {Hunter Dinkins and Andrey Smirnov}
}

@article{KS2,
    author = {Kononov, Yakov and Smirnov, Andrey},
    title = "{Pursuing Quantum Difference Equations II: 3D mirror symmetry}",
    journal = {International Mathematics Research Notices},
    volume = {2023},
    number = {15},
    pages = {13290-13331},
    year = {2022},
    month = {07},
    issn = {1073-7928},
    doi = {10.1093/imrn/rnac196}
}

@article{dinkms1,
       author = {{Dinkins}, Hunter},
        title = "{Symplectic Duality of $T^*Gr(k,n)$}",
      journal = {Mathematical Research Letters},
      year={2021},
      volume={29},
      number={3}
}

@article{msflag,
       author = {{Dinkins}, Hunter
       },
        title = {3d mirror symmetry of the cotangent bundle of the full flag variety},
     journal={Letters in Mathematical Physics},
     year={2022},
     volume={112},
     number={100},
       doi={/10.1007/s11005-022-01593-4}
}

@article{KS1,
author = {Kononov, Yakov and Smirnov, Andrey},
year = {2022},
month = {07},
pages = {},
title = {Pursuing quantum difference equations I: stable envelopes of subvarieties},
volume = {112},
journal = {Letters in Mathematical Physics},
doi = {10.1007/s11005-022-01561-y}
}

@article{RSbows,
       author = {{Rimanyi}, R. and {Shou}, Y.},
        title = "{Bow varieties---geometry, combinatorics, characteristic classes}",
      journal = {Communications in Analysis and Geometry},
      pages={to appear},
         year = 2020,
        month = 12,
archivePrefix = {arXiv},
       eprint = {2012.07814},
 primaryClass = {math.AG}
}

@article{dinksmir4,
author = {Dinkins, Hunter and Smirnov, Andrey},
year = {2022},
month = {07},
pages = {72},
title = {Euler characteristic of stable envelopes},
volume = {28},
journal = {Selecta Mathematica},
doi = {10.1007/s00029-022-00788-w}
}

@article{dinkinselliptic,
    author = {Dinkins, Hunter},
    title = "{Elliptic Stable Envelopes of Affine Type A Quiver Varieties}",
    journal = {International Mathematics Research Notices},
    year = {2022},
    month = {07},
    doi = {10.1093/imrn/rnac198}
}

@phdthesis{dinkinsthesis,
    author = "Dinkins, Hunter",
    title = "{Exotic Quantum Difference Equations and Integral Solutions}",
    eprint = "2205.01596",
    archivePrefix = "arXiv",
    primaryClass = "math.AG",
    doi = "10.17615/4h4e-sj63",
    school = {University of North Carolina Chapel Hill},
    year = {2022}
}

@misc{KZins,
       author = {{Koroteev}, Peter and {Zeitlin}, Anton M.},
        title = "{3d Mirror Symmetry for Instanton Moduli Spaces}",
         year = 2021,
        month = may,
archivePrefix = {arXiv},
       eprint = {2105.00588},
 primaryClass = {math.AG},
}

@book{KZlectures,
  title={Lectures on Representation Theory and Knizhnik-Zamolodchikov Equations},
  author={Pavel Etingof and Igor B. Frenkel and Alexander A. Kirillov},
  year={1998},
volume={58},
publisher={American Mathematical Society},
series={Mathematical Surveys and Monographs}
}

@article{Naktp,
       author = {{Nakajima}, Hiraku},
        title = "{Quiver varieties and tensor products}",
      journal = {Inventiones Mathematicae},
         year = 2001,
        month = nov,
       volume = {146},
       number = {2},
        pages = {399-449},
          doi = {10.1007/PL00005810}
}

@article{varagnolo,
      title={Quiver Varieties and Yangians}, 
      author={Michela Varagnolo},
      year={2000},
      journal={Letters in Mathematical Physics},
number = {53},
pages= {273-283}
}

@misc{smirnovrationality,
      title={Rationality of capped descendent vertex in $K$-theory}, 
      author={Andrey Smirnov},
      year={2016},
      eprint={1612.01048},
      archivePrefix={arXiv},
      primaryClass={math.AG}
}

@misc{RB,
       author = {{Botta}, Tommaso Maria and {Rimanyi}, Richard},
        title = "{Bow varieties: Stable envelopes and their 3d mirror symmetry}",
         year = 2023,
        month = aug,
        pages = {arXiv:2308.07300},
archivePrefix = {arXiv},
       eprint = {2308.07300},
 primaryClass = {math.AG},
}

\newpage

\noindent
Hunter Dinkins\\
Department of Mathematics,\\
Northeastern University,\\
Boston, MA USA\\
h.dinkins@northeastern.edu 

\end{document}